\documentclass[a4paper,11pt]{elsarticle}
\usepackage{amscd,amsmath,amsthm,amsfonts}
\usepackage{graphicx}
\usepackage{amssymb,color,amsbsy}

\usepackage[matrix,arrow]{xy}
\usepackage{mathabx}

\usepackage{tikz}
\usepackage{caption}
\usepackage{amsthm}
\usepackage{multicol}
\usepackage{subcaption}

\usepackage[pdftex]{hyperref}
\hypersetup{colorlinks=false}

\newtheorem{theorem}{Theorem}[section]

\newtheorem{algorithm}[theorem]{Algorithm}

\newtheorem{corollary}[theorem]{Corollary}

\newtheorem{definition}[theorem]{Definition}

\newtheorem{lemma}[theorem]{Lemma}

\newtheorem{proposition}[theorem]{Proposition}

\DeclareMathAlphabet{\mathpzc}{OT1}{pzc}{m}{it}

\newcommand{\Supp}[2]{\operatorname{\mathpzc{Supp}}_{#1}(#2)}
\newcommand{\ISupp}[1]{\operatorname{\mathpzc{IntSupp}}(#1)}
\newcommand{\N}[1]{\operatorname{\mathpzc{N}}(#1)}
\newcommand{\Rank}[1]{\operatorname{\mathpzc{R}}(#1)}
\newcommand{\Core}[1]{\operatorname{\mathpzc{Core}}(#1)}

\newcommand{\ConnE}[1]{\operatorname{\mathpzc{ConnE}}(#1)}
\newcommand{\BondE}[1]{\operatorname{\mathpzc{BondE}}(#1)}

\newcommand{\entrada}[2]{\operatorname{in}(#1 \leftarrow #2)}
\newcommand{\salida}[2]{\operatorname{out}(#1 \rightarrow #2)}

\newcommand{\supp}[2]{\operatorname{supp}_{#1}(#2)}
\newcommand{\nulidad}[1]{\operatorname{null}(#1)}
\newcommand{\rank}[1]{\operatorname{rk}(#1)}
\newcommand{\core}[1]{\operatorname{core}(#1)}

\newcommand{\down}[3]{{#1}\!\downharpoonleft_{\scalebox{0.5}{#3}}^{\scalebox{0.5}{#2}}}
\newcommand{\up}[3]{{#1}\!\upharpoonleft_{\scalebox{0.5}{#3}}^{\scalebox{0.5}{#2}}}

\journal{\ldots}

\begin{document}


\begin{frontmatter}
	\title{S-Trees}
	
	\author[daj]{Daniel A Jaume\corref{cor1}}
	\ead{djaume@unsl.edu.ar}
	
	\author[gm]{Gonzalo Molina}
	\ead{lgmolina@unsl.edu.ar}
	
	\author[rs]{Rodrigo Sota}
	\ead{rodrigo.a.sota@gmail.com}
	
	\cortext[cor1]{Corresponding author: Daniel A. Jaume}
	
	\address[daj]{	Departamento de Matem\'{a}ticas.
		Facultad de Ciencias F\'{\i}sico-Matem\'{a}ticas y Naturales. Universidad Nacional de San Luis.
		1er. piso, Bloque II, Oficina 54. Ej\'{e}rcito de los Andes 950.
		San Luis, Rep\'{u}blica Argentina.
		D5700HHW.}
	

\begin{abstract}
	In this paper two new graph operations are introduced, and with them the S-trees are studied in depth. This allows to find \(\{-1,0,1\}\)-basis for all the fundamental subspaces of the adjacency matrix of any tree, and to understand in detail the matching structure of any tree. 
\end{abstract}


\end{frontmatter}



\section{Introduction}

Collatz and Sinogowitz \cite{von1957spektren} and Schwenk and Wilson \cite{longuet1950some} posed the problem of characterizing all singular and nonsingular graphs. This problem is very difficult. At present, only some particular cases are know, see the excellent survey \cite{gutman2011nullity}. On the other hand, this problem is very interesting in chemistry, because, as has been shown in Longuet-higgins \cite{longuet1950some}, the occurrence of a zero eigenvalue of a bipartite graph (corresponding to an alternant hydrocarbon) indicates the chemical instability of the molecule which such graph represents.

In recent years there has been a growing interest on the structure of null space of incidence and adjacency matrices of graphs, see \cite{akbari2006kernels}, \cite{akbari20061}, \cite{sander2009simply}.
In particular the existence of \(\{−1, 0, 1\}\)-bases has attracted more
attention, since these kind of bases are usually sparse and easy to handle from the computational point of view. In 2004, in his PhD thesis Sander proposed the following Conjecture: The null space of the adjacency
matrix of every forest has a \(\{−1, 0, 1\}\)-basis, see \cite{sander2004eigenspace}. This was fully answered in \cite{sander2005simply} and \cite{akbari20061}.

In this paper the S-trees are study in deep. They were introduced by Jaume and Molina in \cite{molina_jaume2016tree} in order to settle the Null Decomposition of trees, see Section \ref{SecStrees}.  They allow to fully understand the matching structure, and encodes the null space (and then the nullity) of any tree, thus S-tree is a new way of characterizing all singular trees. They allow to give another proof of the Sander's problem. The S-part of any tree give an easy way to build \(\{-1,0,1\}\)-bases for the null space and the range of any tree. The basis that we found for the null space of a tree are pretty similar to the basis given by the FOX algorithm of Sander and Sander \cite{sander2005simply}. But our constructive linear algebra approach is different from \cite{sander2005simply} and from the existential approach of \cite{akbari20061}. It is interesting to note that recently, in \cite{akbari2004ranks}, a more general question has
been asked: For which graphs is there a \(\{−1, 0, 1\}\)-basis for the null space?

This paper is organized as follows. In Section 2 we set up notation and terminology, and also review some of the standard facts on graphs. In Section 3 we give the definition of S-tree, the null decomposition, and a brief summary of the relevant results of \cite{molina_jaume2016tree}. There is just one result is new in this section: Theorem \ref{Trknulloftrees}. Section 4 deal with two new graph operations over trees: stellare and S-coalescence. They are closed on S-trees and allow us to think about S-trees without linear algebra. Section 5 is concerned with the notion of S-atoms, they are the bricks with which all the S-trees are made. Section 6 is concerned with the notion of S-basic, they encode all the information needed to obtain a \(\{-1,0,1\}\)-basis of a tree. Finally in Section 7 we study the fundamental subspaces associated to the adjacency matrix of any tree and we give \(\{-1,0,1\}\)-bases for all of them. 

\section{Basics and notation}
The material in this section is standard. Our recommendation is that you begin to read the paper  by the next section, and return to this section whether you have any doubt about notation.

In this work we will only consider finite, loopless, simple, connected graphs. Let \(G\) be  a graph:
\begin{enumerate}
	\item \(V(G)\) is the set of vertices of \(G\), and \(v(G)=|V(G)|\) denote its cardinality.
	\item Let \(S \subset V(G)\), the subgraph induced by \(S\) in \(G\) is denoted by  \(G\langle S \rangle\). With \(G-S\) we denote the subgraph of \(G\) obtained by deleting the vertices in \(S\). We usually just write \(G-v\) instead of \(G-\{v\}\), for \(v \in V(G)\).
	\item \(E(G)\) is the set of edges of \(G\), and \(e(G)=|E(G)|\), its cardinality.
	\item Let \(u \sim v\) denote that two vertices \(u\) and \(v\) of \(G\) are neighbors: \(\{u,v\} \in E(G)\).
	\item Let \(N_{G}(v)\) denote the set of neighbors of \(v\) in \(G\), if \(G\) is clear from the context we just write \(N(v)\). The closed neighborhood of \(v\) is \(N[v]=N(v)\cup\{v\}\). For \(S \subset V(G) \) the closed neighborhood of \(S\) is
	\[
	N\left[ S \right]:= \bigcup_{u \in S}N\left[u\right]
	\]
	\item Let \(\deg(v)\), the degree of \(v\), denote the cardinality of \(N(v)\).
	\item A vertex \(v\) of \(G\) is a pendant vertex if \(\deg(v)=1\).
	\item \(A(G)\) is the adjacency matrix of \(G\), if \(G\) is clear from the context we drop \(G\) and just write \(A\).
	\item The rank of \(G\) is the rank of its adjacency matrix: \(\rank{G}:=\rank{A(G)}\).
	\item The null space of \(G\) is the null space of its adjacency matrix: \(\N{G}:=\N{A(G)}\).
	\item The nullity of \(G\) is the nullity of its adjacency matrix: \(\nulidad{G}:=\nulidad{A(G)}\).
	\item The spectrum of \(G\) is the set of different eigenvalues of \(A(G)\), and it will be denoted by \(\sigma (G)\). Given an eigenvalue \(\lambda \in \sigma (G) \) the eigenspace associated to \(\lambda\), denoted by \(\mathcal{E}_{\lambda}(G)\), will be called \(\lambda\)-eigenspace of \(G\).
	\item A set \(S\) of vertices is an independent set if no two vertices in \(S\) are adjacent. \(\alpha(G)\) denote the independence number of \(G\), the cardinality of a maximum independent set. A set\( S \subset V(G)\) is a total dominating set og \(G\) if every vertex in \(G\) is adjacent to
	some vertex of \(S\). 
	\item A matching \(M\) in  \(G\) is a set of pairwise non-adjacent edges; that is, no two edges in \(M\) share a common vertex. A vertex is saturated by a matching, if it is an endpoint of one of the edges in the matching. Otherwise the vertex is non-saturated. The  matching number of \(T\), denoted \(\nu(T)\), is the cardinality of a maximum matching. The number of maximum matchings in \(G\) is \(m(G)\). 
	\item A set \(S \subset V \) is a dominating set of \(G\) if each vertex in \(V(G)\) is either in \(S\) or is adjacent to a vertex in \(S\). The domination number \(\gamma (G)\) is the minimun cardinality of a dominating set of \(G\).
	\item A graph \(G\) is bipartite if its vertices can be partitioned in two sets in such a way that no edge join two vertices in the same set.
\end{enumerate}

Let \(\theta\) denote the null-vector of a given vector space. Given matrix \(A\), its transpose will be denoted by \(A^t\). As usual, \([k]=\{1,\cdots,k\}\). We usually think vectors as functions from the set of vertices of a graph to a field (typically \(\mathbb{R}\)), thus, given a graph \(G\),  \(\mathbb{R}^{G}\) denote the vector space of all the real functions from \(V(G)\) to \(\mathbb{R}\), this vector space has dimension \(v(G)\).

\section{S-trees and null decomposition of trees}  \label{SecStrees}

The S-trees and the null decomposition of trees were introduced in \cite{molina_jaume2016tree}. Here a brief resume of the results in \cite{molina_jaume2016tree} that are needed for the present work.

\begin{definition}[\cite{molina_jaume2016tree}]
	Let \(x\) be a vector of \(\mathbb{R}^{n}\), the \textbf{support} of \(x\) is 
	\[
	\Supp{\mathbb{R}^{n}}{x}:= \lbrace v \in [n] : x_{v} \neq 0 \rbrace 
	\]
	Let \(S\) be a subset of \(\mathbb{R}^n\), then the support of \(S\) is
	\[
	\Supp{\mathbb{R}^{n}}{S}:= \bigcup_{x \in S}  \Supp{\mathbb{R}^{n}}{x}
	\]
	The cardinality of an support is denoted by \(\supp{\mathbb{R}^{n}}{*}:=|\Supp{\mathbb{R}^{n}}{*}| \).
\end{definition}
 
 In this paper only the support of null space  of a tree (i.e. the null space of the adjacency matrix of trees) is considered. Thus \(\Supp{}{T} \) stands for \(\Supp{T}{\N{T}}\). Similarly, \(\supp{}{T}\) stands for \(\supp{T}{\N{T}}\).

\begin{lemma}[Lemma 7, \cite{nylen1998null}] \label{vecsupp}
	Let \(W\) be a subspace of \(\mathbb{R}^{n}\) with dimension \(d>0\). Then there exists a basis \(\{w_{1},\dots,w_{d}\}\) of \(W\) satisfying 
	\[
	\Supp{\mathbb{R}^{n}}{w_{i}}=\Supp{\mathbb{R}^{n}}{W}
	\]
	for all \(i=1, \dots, d\).
\end{lemma}

\begin{definition}[\cite{molina_jaume2016tree}]
	A tree \(S\) is an \textbf{S-tree} if \(
	N\left[   \Supp{}{S} \right]=V(S) \).
\end{definition}

The tree \(T\) in Figure ~\ref{Fig2} is an S-tree. 

\tikzstyle{every node}=[circle, draw, fill=white!,
inner sep=0.1pt, minimum width=11pt]

\begin{figure}[h]
	\centering
	\begin{tikzpicture}[thick,scale=0.2]%
	\draw 
	(0,0) node{1}
	(-5,0) node[fill=black!]{\textcolor{white}{3}} -- (0,0)
	(0,-5) node[fill=black!]{\textcolor{white}{2}} -- (0,0)
	(0,5) node[fill=black!]{\textcolor{white}{4}} -- (0,0)
	(5,0) node[fill=black!]{\textcolor{white}{5}} -- (0,0)
	(10,0) node{6} -- (5,0)
	(10,5) node[fill=black!]{\textcolor{white}{7}} -- (10,0)
	(10,-5) node[fill=black!]{\textcolor{white}{8}} -- (10,0);
	
	\end{tikzpicture}
	\caption{T is an S-tree}
	\label{Fig2}
\end{figure}
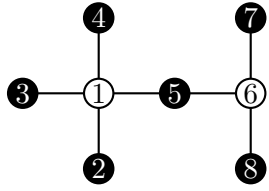

The adjacency matrix of \(T\), given in Figure ~\ref{Fig2} is:
\[
A(T)=
\left[
{ 
	\begin{array}{cccccccc}
	0 & 1 & 1 & 1 & 1 & 0 & 0 & 0\\
	1 & 0 & 0 & 0 & 0 & 0 & 0 & 0\\
	1 & 0 & 0 & 0 & 0 & 0 & 0 & 0\\
	1 & 0 & 0 & 0 & 0 & 0 & 0 & 0\\
	1 & 0 & 0 & 0 & 0 & 1 & 0 & 0\\
	0 & 0 & 0 & 0 & 1 & 0 & 1 & 1\\
	0 & 0 & 0 & 0 & 0 & 1 & 0 & 0\\
	0 & 0 & 0 & 0 & 0 & 1 & 0 & 0\\
	\end{array}
}
\right]
\]
It is obvious that \(0 \in \sigma (T)\), and that \(\{(0,1,0,0,-1,0,0,1)^{t},(0,0,1,0,-1,0,0,1)^{t}\newline ,(0,0,0,1,-1,0,0,1)^{t},(0,0,0,0,0,0,1,-1)^{t}\}\) is a base of the \(0\)-eigenspace of \(T\).  Thus \(\Supp{}{T}= \{2,3,4,5,7,8\}
\) and 
\[
N \left[ \Supp{}{T} \right]= N \left[ \{2,3,4,5,7,8\} \right]=V(T)
\]


\begin{definition}[\cite{molina_jaume2016tree}]
	Let \(S\) be an S-tree. The \textbf{core} of \(S\), denoted by \(\Core{S}\), is defined to be  \(S\):
	\[
	\Core{S}:=N(\Supp{}{S})
	\]
\end{definition}

A vertex \(v\) is called a \textbf{core-vertex} of an S-tree \(S\) if \(v \in \Core{S}\). We will denote by  \(\core{S}\) the cardinality of \(\Core{S}\).

The tree of order 1 is an S-tree, the unique S-tree without core.

%
%
%
%

\begin{theorem}[\cite{molina_jaume2016tree}]\label{S-tree-Independent}
	Let \(S\) be an S-tree, then \(\alpha(S)=\supp{}{T}\) and \(\Supp{}{S}\) is the unique maximum independent set of \(T\). 
\end{theorem}

We denote with \(\mathcal{M}(G)\) the set of all maximum matchings of a graphs \(G\), and \(m(G):=|\mathcal{M}(G)|\).

\begin{theorem}[\cite{molina_jaume2016tree}] \label{MatchingMaximumStrees}
	If \(S\) is a S-Tree, then \(\nu(S)= \core{S} \). Furthermore \(\Core{S}\) is the unique minimum vertex cover of \(S\), and for all \(M \in \mathcal{M}(S)\), and for all \(e \in M\) we have that \(|e \cap \Core{S}|=1\). 
\end{theorem}

Let \(M\) be a matching in a graph \(G\), in what follows, \(V(M)\) stands for the set of vertices of \(M\). 

\begin{lemma}[\cite{molina_jaume2016tree}] \label{SuppNSaturated}
	Let \(S\) be an S-tree. If \(v \in \Supp{}{T}\), then there exists \(M \in \mathcal{M}(S)\) such that \(v \notin V(M)\). 
\end{lemma}


%
%

A tree \(N\) is non-singular tree,  \textbf{N-tree} for short, if its adjacency matrix \(A(N)\) is invertible. Thus \(N\) is a N-tree if and only if \(N\) has a perfect matching, see \cite{bapat2010graphs}. By 	K\"{o}nig-Egerv\'{a}ry Theorem, if \(N\) is a N-tree, \(\alpha(N)=\frac{v(N)}{2}\).

\begin{definition}[\cite{molina_jaume2016tree}]
	Let \(T\) be a tree. The \textbf{S-Set} of T, denoted by \(\mathcal{F}_{S}(T)\), is defined to be the set of connected components of the forest induced by the closed neighbor of \(\Supp{}{T}\) in \(T\):
	\[
	\mathcal{F}_{S}(T):=\{S \; : \; S \text{ is a connected component of } T\left\langle N\left[ \Supp{}{T} \right] \right\rangle \}
	\]
	The \textbf{N-set} of T, denoted by \(\mathcal{F}_{N}(T)\), is defined to be the set of connected components of the forest that memains after taking away the S-set:
	\[
	\mathcal{F}_{N}(T):= \{N \; :\; N \text{ is a connected component of } T \backslash \mathcal{F}_{S}(T)\}
	\]
	The pair of sets \((\mathcal{F}_{S}(T),\mathcal{F}_{N}(T))\)  is called the \textbf{null decomposition} of \(T\).
\end{definition}

We think the sets \(\mathcal{F}_{S}(T)\) and \(\mathcal{F}_{N}(T)\) as forests. Thus \(E(\mathcal{F}_{S}(T))\) is the set of edges of all trees in \(\mathcal{F}_{S}(T)\). With \(V(\mathcal{F}_{N}(T))\) we denote the set of vertices of all trees in \(\mathcal{F}_{N}(T)\). Even some times we talk about the S-forest \(\mathcal{F}_{S}(T)\) instead of S-set, and the about the N-forest \(\mathcal{F}_{N}(T)\) instead of N-set.



The tree of order 1 is never part of the null decomposition of any other tree.

By construction, for any tree \(T\) there is a set of edges in \(E(T)\) that are neither edges of \(\mathcal{F}_{S}(T)\) nor of \(\mathcal{F}_{N}(T)\). They are called the \textbf{connection edges} of \(T\), and the set of all of them is denoted \(\ConnE{T}\):
\[
\ConnE{T}:=E(T) \setminus \left(E(\mathcal{F}_{S}(T)) \cup E(\mathcal{F}_{N}(T))  \right) 
\]
Note that for all \(e \in \ConnE{T}\), we have that \(e \cap V(\mathcal{F}_{S}(T)) \neq \emptyset \) and \(e \cap V(\mathcal{F}_{N}(T)) \neq \emptyset \). Given \(S \in \mathcal{F}_{S}(T) \) and \(N \in \mathcal{F}_{N}(T) \) we say that \(S\) and \(N\) are \textbf{adjacent parts}, denoted by \(S \sim N\), if there exists \(e \in \ConnE{T}\) such that  \(e \cap V(S) \neq \emptyset \) and  \(e \cap V(N) \neq \emptyset \).

For example, consider the tree in Figure ~\ref{fig_el}, its S-forest is
\[
\mathcal{F}_{S}(T)=\{T \langle \{1,2,3\}\rangle, T \langle \{4,5,6,7,8\}\rangle, T \langle \{9,10,11,12\}\rangle\}=\{S_{1}, S_{2},S_{3}\}
\]
and its N-forest is
\[
\mathcal{F}_{N}(T)=\{T \langle \{13,14\}\rangle, T \langle \{15,16,17,18 \}\rangle\}=\{N_{1}, N_{2}\}
\]


\begin{figure}[h]
	\centering
	\begin{tikzpicture}[thick,scale=0.2]%
	
	\draw 
	(0,0) node[fill={rgb:black,1;white,3}]{2}
	(5,0) node[fill={rgb:black,1;white,6}]{1} -- (0,0)
	(10,0) node{13} -- (5,0)
	(15,0) node{14} -- (10,0)
	(20,0) node[fill={rgb:black,1;white,6}]{4} -- (15,0)
	(5,5) node[fill={rgb:black,1;white,3}]{3} -- (5,0)
	(15,-5) node{16} -- (10,-5)
	(20,-5) node{15} -- (15,-5)
	(25,0) node[fill={rgb:black,1;white,3}]{7} -- (20,0)
	(25,-5) node[fill={rgb:black,1;white,6}]{5} -- (25,0)
	(10,-5) node[fill={rgb:black,1;white,6}]{9} -- (10,0)
	(5,-5) node[fill={rgb:black,1;white,3}]{10} -- (10,-5)
	(10,-10) node[fill={rgb:black,1;white,3}]{12} -- (10,-5)
	(15,-15) node{18} -- (15,-10)
	(15,-10) node{17} -- (15,-5)
	(25,-10) node[fill={rgb:black,1;white,3}]{8} -- (25,-5)
	(20,5) node[fill={rgb:black,1;white,3}]{6} -- (20,0)
	(10-3.5355,-8.5355) node[fill={rgb:black,1;white,3}]{11} -- (10,-5);

	\draw [densely dashdotted,very thick] (0,-2) -- (4.8,-2);
	\draw [densely dashdotted,very thick] (5,-2) arc (-90:0:2);
	\draw [densely dashdotted,very thick] (0,2) arc (90:270:2);
	\draw [densely dashdotted,very thick] (7,5) -- (7,0);
	\draw [densely dashdotted,very thick] (3,4.9) -- (3,4.3);
	\draw [densely dashdotted,very thick] (0.3,2) -- (1,2);
	\draw [densely dashdotted,very thick] (1,2) arc (-90:0:2);
	\draw [densely dashdotted,very thick] (3,5) arc (180:0:2);
	
	\draw [densely dashdotted,very thick] (27,0) -- (27,-10);
	\draw [densely dashdotted,very thick] (27,-10) arc (0:-180:2);
	\draw [densely dashdotted,very thick] (27,0) arc (0:90:2);
	\draw [densely dashdotted,very thick] (23,-10) -- (23,-4);
	\draw [densely dashdotted, very thick] (27,0) -- (27,-10);
	\draw [densely dashdotted, very thick] (22,5) arc (0:180:2);
	\draw [densely dashdotted, very thick] (20,-2) arc (-90:-180:2);
	\draw [densely dashdotted, very thick] (23,-4) arc (0:90:2);
	\draw [densely dashdotted, very thick] (24,2) arc (-90:-180:2);
	\draw [densely dashdotted, very thick] (18,0) -- (18,5);
	\draw [densely dashdotted, very thick] (20,-2) -- (21,-2);
	\draw [densely dashdotted, very thick] (25,2) -- (24,2);
	\draw [densely dashdotted, very thick] (22,5) -- (22,4);
	
	\draw [densely dashdotted, very thick] (10,-3) -- (5,-3);
	\draw [densely dashdotted, very thick] (12,-5) -- (12,-10);
	\draw [densely dashdotted, very thick] (12,-5) arc (0:90:2);
	\draw [densely dashdotted, very thick] (5,-3) arc (90:270:2);
	\draw [densely dashdotted, very thick] (12,-10) arc (0:-180:2);
	\draw [densely dashdotted, very thick] (5,-7) arc (135:315:2);
	
	\draw (0,5) node[white]{\textcolor{black}{$S_1$}};
	\draw (25,5) node[white]{\textcolor{black}{$S_2$}};
	\draw (2.5,-10) node[white]{\textcolor{black}{$S_3$}};
	
	\draw [dotted] (10,-2) -- (15,-2);
	\draw [dotted] (10,2) -- (15,2);
	\draw [dotted] (15,2) arc (90:-90:2);
	\draw [dotted] (10,-2) arc (-90:-270:2);
	
	\draw [dotted] (15,-3) -- (20,-3);
	\draw [dotted] (17,-9) -- (17,-15);
	\draw [dotted] (13,-15) -- (13,-5);
	\draw [dotted] (15,-3) arc (90:180:2);
	\draw [dotted] (20,-3) arc (90:-90:2);
	\draw [dotted] (17,-15) arc (0:-180:2);
	\draw [dotted] (19,-7) arc (90:180:2);
	\draw [dotted] (20,-7) -- (19,-7);

	\draw (12.5,4) node[white]{\textcolor{black}{$N_1$}};
	\draw (19,-12.5) node[white]{\textcolor{black}{$N_2$}};

	\end{tikzpicture}
	\caption{Null decomposition of a tree \(T\)}
	\label{fig_el}
\end{figure}
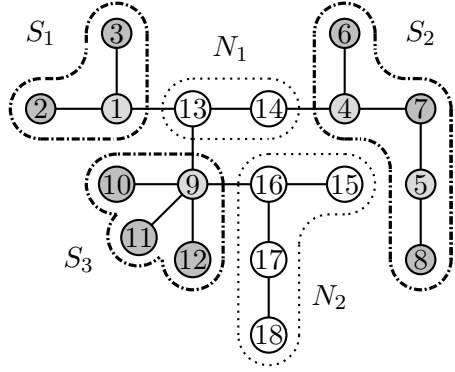

Let \(x\) be a vector whose components are indexed by the vertices of a tree \(T\), and \(S \subset  V(T)\), then \(\down{x}{T}{S}\) is the vector that we obtain when \(x\) is restricted to the coordinates associated to \(S\). For example, let 
\[
x^{t}=(1,2,\dots,18)
\]
be a vector whose coordinates are indexed by the vertices of the tree \(T\) in Figure ~\ref*{fig_el}, then  
\[
(\down{x}{T}{S{\small 2}})^{t} = (4,5,6,7,8)
\]

Let \(U,W,\) and \(G\) be graphs such that \(U \leq W \leq G\), and let \(x \in \mathbb{R}^{G}\) then
\(
\down{\left(\down{x}{G}{W}\right)}{W}{U}= \down{x}{G}{U }
\).

For any \(n\)-graph \(G\) and any \(m\)-induced-sugbraph \(H \leq G\), given \(x \in \mathbb{R}^{H}\) whose coordinates are indexed by \(V(H)\), by \(\up{x}{G}{H}\) we denoted the lift of \(x\) to a vector of \(\mathbb{R}^{G}\) and it is carried out in the following way:
\begin{itemize}
	\item for any \(u \in V(T) \setminus V(H)\), \((\up{x}{G}{H})_{u}=0\).
	\item for any \(u \in V(H)\), \((\up{x}{G}{H})_{u}=x_{u}\).
\end{itemize}
For example, let 
\[
x^{t}=(1,2,3,4)
\]
be a vector of \(\mathbb{R}^{S_{3}}\). Consider the tree \(T\) and the subtree \(S_{3}\), see Figure ~\ref{fig_el}, then  
\[
(\up{x}{T}{S{\small 3}})^{t} = (0,0,0,0,0,0,0,0,0,1,2,3,4,0,0,0,0,0,0)
\]

\begin{theorem}[\cite{molina_jaume2016tree}]
	Let \(T\) be a tree. The S-set \(\mathcal{F}_{S}(T)\) of \(T\) is a set of S-trees.
\end{theorem}

Let \(T\) be a tree, the \textbf{core} of \(T\), denoted by \(\Core{T}\), is the set
\[
\Core{T}:=\bigcup_{S \in \mathcal{F}_{S}(T)} \Core{S}=N(\Supp{}{T})
\]
As before, \(\core{T}=|\Core{T}|\).


\begin{proposition}[\cite{molina_jaume2016tree}] \label{C_null}
	Let \(T\) be a tree, then 
	\[
	\N{T}=\bigoplus_{S \in \mathcal{F}_{S}(T)} \up{\N{S}}{$T$}{$S$}
	\]
	Hence \(\nulidad{T}=\sum_{S \in \mathcal{F}_{S}(T)} \nulidad{S}\).
\end{proposition}



From \cite{bevis1995ranks} we know that the rank of the adjacency matrix of a tree equals two times the matching number of \(T\): \(\rank{A}=2\nu(T)\).

%

\begin{theorem}[\cite{molina_jaume2016tree}]
	Let \(T\) be a tree, the N-set \(\mathcal{F}_{N}(T)\) of \(T\) is a set of N-trees.
\end{theorem}

\begin{corollary}[\cite{molina_jaume2016tree}] \label{CoroMatching}
	Let  \(T\) be a tree, and a maximum matching \(M \in \mathcal{M}(T)\). Then \(M \cap \ConnE{T} = \emptyset \), and for all \(e \in M\) \(|e \cap \Core{T}|\leq 1\).
\end{corollary}

\begin{corollary}[\cite{molina_jaume2016tree}] \label{nualpha_coresupp}
	Let \(T\) be a tree, then
	\begin{align*}
	\nu(T) = & \core{T}	+\frac{v(\mathcal{F}_{N}(T))}{2}\\
	\alpha(T)= & \supp{}{T}+\frac{v(\mathcal{F}_{N}(T))}{2}
	\end{align*}
\end{corollary}

\begin{corollary}[\cite{molina_jaume2016tree}]
	Let \(T\) be a tree, then
	\[
	m(T)=\prod_{S \in\mathcal{F}_{S}(T) }m(S)
	\]
\end{corollary}

%
%
%

Now a new result.

\begin{theorem} \label{Trknulloftrees}
	Let \(T\) b a tree. Then
	\begin{align*}
	\rank{T} &=2\core{T}+v(\mathcal{F}_{N}(T))\\
	\nulidad{T} & = \supp{}{T}-\core{T}
	\end{align*}
\end{theorem}

\begin{proof}
	From \cite{bevis1995ranks} we know that the rank of the adjacency matrix of \(T\) equals
	two times the matching number of \(T\), and from \cite{molina_jaume2016tree} we know that \(\nu(T)=\core{T}+\frac{1}{2}v(\mathcal{F}_{N}(T)) \). Hence \(\rank{T} = 2\core{T} + v(\mathcal{F}_{N}(T))\). On the other hand, again from \cite{molina_jaume2016tree}, we know that \(n\), the order of \(T\), equals \(v(\mathcal{F}_{S}(T))+v(\mathcal{F}_{N}(T))\), and that \(v(\mathcal{F}_{S}(T))=\core{T}+\supp{}{T}\). As \(\nulidad{T}=n-\rank{T}\), we obtain \(\nulidad{T} = \supp{}{T}-\core{T}\).
\end{proof}

\begin{corollary} \label{jaumelario}
	Let \(T\) be a tree. Then \(\nu(T)=\alpha(T)-\nulidad{T}\).
\end{corollary}

\begin{proof}
	By Corollary \ref{nualpha_coresupp} and Theorem \ref{Trknulloftrees}
	\begin{align*}
	\nu(T) & = \core{T}	+\frac{v(\mathcal{F}_{N}(T))}{2}\\
	{} & =  \core{T} + \alpha(T) -\supp{}{T}\\
	{} & = \alpha(T)-\nulidad{T}
	\end{align*}
\end{proof}

\section{Graph operations closed over S-trees}


Now we will define two graph operations under which S-trees are closed. These operations are important because they allow to think about S-trees without finding null spaces.

	\subsection{Stellare}
	

\begin{definition}
Let \(G\) be a labeled graph of order \(n\), with labels \([n]\).
The \(*(k_{1},\dots,k_{n})\)-\textbf{stellare} of \(G\) is the graph obtained by adding \(k_{i}\geq 2\) pendant vertices to vertex \(i\) of \(G\).
With \(*G\) we denote an arbitrary stellare of \(G\).
\end{definition}

\begin{figure}[h] 
	\centering
	\begin{subfigure}[c]{0.3\textwidth}
		\begin{tikzpicture}[thick,scale=0.2]
		
		\draw 
		(0,0) node{1}
		(0,5) node{2} -- (0,0)
		(5,0) node{3} -- (0,0)
		(5,5) node{4} -- (5,0)
		(10,0) node{5} -- (5,0)
		(15,0) node{6} -- (10,0);
		
		\end{tikzpicture}
		\caption{A tree $T$}
		
	\end{subfigure}
	
	\hfill
	
	\begin{subfigure}[c]{\textwidth}
		\centering
		\begin{tikzpicture}[thick,scale=0.2]
		\draw 
		(0,0) node[fill=gray!]{\tiny{1,0}}
		(0,5) node[fill=gray!]{\tiny{2,0}} -- (0,0)
		(5,0) node[fill=gray!]{\tiny{3,0}} -- (0,0)
		(5,5) node[fill=gray!]{\tiny{4,0}} -- (5,0)
		(10,0) node[fill=gray!]{\tiny{5,0}} -- (5,0)
		(15,0) node[fill=gray!]{\tiny{6,0}} -- (10,0)
		
		(0,10) node{\tiny{2,2}} -- (0,5)
		(3.535,8.535) node{\tiny{2,3}} -- (0,5)
		(-3.535,8.535) node{\tiny{2,1}} -- (0,5)
		
		(-5,0) node{\tiny{1,2}} -- (0,0)
		(-3.535,-3.535) node{\tiny{1,2}} -- (0,0)
		(-3.535,3.535) node{\tiny{1,1}} -- (0,0)
		(0,-5) node{\tiny{1,4}} -- (0,0)
		
		(5,-5) node{\tiny{3,2}} -- (5,0)
		(8.535,-3.535) node{\tiny{3,3}} -- (5,0)
		(1.865,-3.535) node{\tiny{3,1}} -- (5,0)
		
		(20,0) node{\tiny{6,2}} -- (15,0)
		(18.535,-3.535) node{\tiny{6,3}} -- (15,0)
		(18.535,3.535) node{\tiny{6,1}} -- (15,0)
		
		(13.535,3.535) node{\tiny{5,1}} -- (10,0)
		(13.535,-3.535) node{\tiny{5,2}} -- (10,0)
		
		(8.535,8.236) node{\tiny{4,1}} -- (5,5)
		(8.535,2) node{\tiny{4,2}} -- (5,5);
		
		\end{tikzpicture}
		\caption{$*(4,3,3,2,2,3)T$}
	\end{subfigure}
	\caption{$*(4,3,3,2,2,3)T$ is an stellare tree of $T$}
	\label{Fig1}
\end{figure}
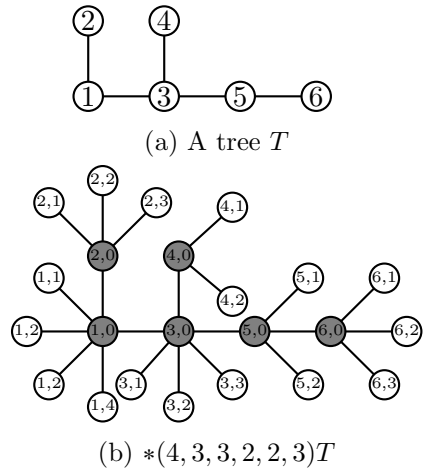

%
%
%
%
%
%
%
%
%
%
%

\begin{theorem} \label{T1Stellare}
Let \(T\) be a tree, then \(*T\) is an S-tree, \(\Core{*T}=V(T)\), and \(\Supp{}{*T}=V(*T) \setminus V(T)\).
\end{theorem}

\begin{proof}
Take \(v \in V(*T) \setminus V(T) \), by definition of stellare there exists \(u \in V(T)\) and \(w \in V(*T) \) such that \(u \sim v \), the vertices \(v,w \) are neighbors of \(u\) in \(*T\), and \(w \notin V(T)\). Let \(x\) be a vector of \(\mathbb{R}^{*T}\), such that
\[
x_{i}=\left\lbrace
\begin{array}{rl}
1&  \text{ if} \; i = v,\\
-1 &   \text{ if} \; i=w,\\
0 &  \text{ otherwise.}
\end{array}
\right.
\]
Clearly \(A(*T)x=\theta\). Hence \(V(*T) \setminus V(T) \subset \Supp{}{*T}\). By stellare definition we have that \(N[V(*T) \setminus V(T)]=V(*T)\). Hence we conclude that \(*T\) is an S-tree, \(\Core{*T}=V(T)\), and \(\Supp{}{*T}=V(*T) \setminus V(T)\).
\end{proof}

\begin{corollary}
	Given a tree \(T\) with labels \([n]\), and a list of \(v(T)\) integers greater or equals to 2: \(k_{1},\dots,k_{v(T)}\). Then
	\begin{enumerate}
		\item \(\nulidad{ *(k_{1},\dots,k_{n})T}= \sum_{i=1}^{n}k_{i} -n \geq n \geq \nulidad{T}\), where equality holds if and only if \(n=1\) and \(k_{1}=2\).
		\item \(\rank{*T}=2n > \rank{T}\).
		\item \(\alpha( *(k_{1},\dots,k_{n})T)=\sum_{i=1}^{n}k_{i} \geq 2n >\alpha(T)\).
		\item \(\nu(*T)=n>\nu(T)\).
		\item \(m(*(k_{1},\dots,k_{n})T)=\prod_{i=1}^{n}k_{i}\).
		\item \(\gamma(*T)=n>\gamma(T)\), and \(V(T)\) is the only minimum (and total, if \(n\geq 2\)) dominating set of any \(*T\). 
	\end{enumerate}
\end{corollary}

\begin{proof}
The statements 1 and 2 follow from Theorem \ref{Trknulloftrees} and Theorem \ref{T1Stellare}. 

Statements 3 follows from Theorem \ref{Trknulloftrees}, Theorem \ref{T1Stellare}, and Theorem \ref{S-tree-Independent}.

Statements 4 and 5 follow from Theorem \ref{Trknulloftrees}, Theorem \ref{T1Stellare}, and Theorem \ref{MatchingMaximumStrees}.

Statement 6 is clear.
\end{proof}

Let \(T\) be a tree, and \(U \subset V(T)\), we set
\begin{equation}\label{eu_notation}
e_{U}:=\sum_{v \in U} e_{v}
\end{equation}
where \(e_{v}\in \mathbb{R}^{T}\), with \((e_{v})_u=1\) if \(u=v\), and \(0\) otherwise.
\begin{definition}
Let \(v\in \Core{T}\), the \(v\)-\textbf{bouquet} of \(T\), denoted by \(R(v)\), is
\[
R(v):=\{u \in \Supp{}{T}: u \sim v \}
\]
\end{definition}

%
%

\begin{lemma}
	Let \(T\) be a tree and \(*T\) an stellare of \(T\). Then the set of vectors of \(\mathbb{R}^{*T}\)
	\[
	\mathcal{B}:=\{e_{v},e_{R(v)} \in \mathbb{R}^{*T} : v \in V(T)\}
	\]
	where \(R(v)\) is the \(v\)-bouquet of \(*T\), is a basis of \(\Rank{*T}\).
\end{lemma}

\begin{proof}
	As \(\rank{*T}=2v(T)\) we only need to prove that the columns of the adjacency matrix of \(*T\) are linear combinations of the vectors of \(\mathcal{B}\). Let \(v \in V(*T)\) with \(A_{v}\) we denote the column of \(A(*T)\) associated to the vertex \(v\). Thus, if \(v \in \Supp{}{*T}\), then 
	\(A_{v}=e_{w}\), where \(w \in \Core{*T}=V(T)\) and \(w \sim v\). If \(v \in \Core{*T}=V(T)\), then 
	\[
	A_{v}=e_{R(v)}+\sum_{\substack{w \in V(T) \\ w \sim v}} e_{w}
	\]
\end{proof}

Given a tree \(T\) with labels \([n]\). The \textbf{stellare labeling} of \(*(k_{1},\dots,k_{n})T\) is the set 
\[
\{(u,w)\; : \; u \in [n], \text{ and } w \in \{0,1, \dots,k_{u} \}\}
\]
where the vertices label with \((u,0)\) are core vertices of \(*T\), i.e. \(V(T)\), and the vertices label with \((u,w)\), where \(w \in \{1,\dots , k_{u}\} \)
are supported vertices of \(*T\), which are neighbors of \(u\). See Figure \ref{Fig1}.

\begin{lemma}
	Let \(T\) be a tree of order \(n\) and \(*(k_{1},\dots,k_{n})T\) an stellare of \(T\). Then the following set of vectors is a basis of null space of \(*(k_{1},\dots,k_{n})T\)
	\[
	\mathcal{B}:=\{\vec{b}(i,j) \in \mathbb{R}^{*T}  : i \in [n], j \in \{2,\dots, k_{i}\}\}
	\]
	where
	\[
	\vec{b}(i,j)_{(u,w)}= 
	\left\lbrace 
	\begin{array}{rl} 
		1  & \text{ if } u=i \text{ and } w=1,\\
		-1 & \text{ if } u=i \text{ and } w=j,\\
		0 & \text{ otherwise}.
	\end{array}
	\right.
	\]
\end{lemma}
\begin{proof} 
	The set \(\mathcal{B}\) is a set of linear independent vectors. A direct computation shows that \(A(*T)\vec{b}=\theta\) for all \(\vec{b} \in \mathcal{B}\). As \(|\mathcal{B}|=\supp{}{*T}-\core{*T}\), then by Theorem \ref{Trknulloftrees} the set \(\mathcal{B}\) is a basis of the null space of \(*(k_{1},\dots,k_{n})T\).
\end{proof}

	\subsection{S-coalescence}
	
	\begin{definition}
Let \(S_{1}, \dots, S_{k}\) be \(k\) S-trees. Let \(v_{i} \in \Supp{}{S_{i}}\), for each \(i \in [k]\). The \textbf{S-coalescence} of \((S_{1},v_{1}), \dots, (S_{k},v_{k})\), denoted by
\[
\displaystyle\bigoasterisk_{i=1}^{k} (S_{i},v_{i})
\]
is the tree obtained by identifying each supported vertex \(v_{i}\) with a new vertex \(v^{*}\). Let \(N_{S_{i}}(v_{i})\) be the neighborhood of \(v_{i}\) in \(S_{i}\), then \(\bigoasterisk_{i=1}^{k}(S_{i},v_{i})\) is the tree with set of vertices 
\[
V\left(\bigoasterisk_{i=1}^{k}(S_{i},v_{i})\right) = \left( \bigcup_{1 \leq i \leq k} (V(S_{i}) \setminus \{v_{i}\}) \right) \cup \{v^{*}\}
\]
and set of edges
\[
E\left(\bigoasterisk_{i=1}^{k}(S_{i},v_{i})\right) = \{ \{u,v^{*}\} : u \in N_{S_{i}}(v_{i})\} \cup \bigcup_{i=1}^{k} E(S_{i}) \setminus \{\{u,v_{i}\} :  u \in N_{T_{i}}(v_{i})\}
\]
\end{definition}


\tikzstyle{every node}=[circle, draw, fill=white!,
inner sep=0.1pt, minimum width=11pt]

\begin{figure}[h]
	\centering
	\begin{subfigure}[c]{0.3\textwidth}
		\begin{tikzpicture}[thick,scale=0.2]%
		
		\draw 
		(0,3) node[fill=gray]{$v_1$}
		(0,8) node{} -- (0,3)
		(0,13) node{} -- (0,8);
		
		\draw[thick,dashed] (0,0) circle (4.5cm);
		
		\draw 
		(2.1213,-2.1213) node[fill=gray]{$v_2$}
		(5.6568,1.4144) node{} -- (2.1213,-2.1213)
		(9.1924,4.95) node{} -- (5.6568,1.4144)
		(5.6568,-5.6568) node{} -- (2.1213,-2.1213)
		(9.1924,-9.1924) node{} -- (5.6568,-5.6568)
		(9.1924,-2.1213) node{} -- (5.6568,-5.6568)
		(2.1213,-9.1924) node{} -- (5.6568,-5.6568)
		(12.7280,-5.6568) node{} -- (9.1924,-9.1924)
		(5.6568,-12.7280) node{} -- (9.1924,-9.1924);
		
		\draw 
		(-2.1213,-2.1213) node[fill=gray]{$v_3$}
		(-5.6568,-5.6568) node{} -- (-2.1213,-2.1213)
		(-9.1924,-9.1924) node{} -- (-5.6568,-5.6568)
		(-9.1924,-2.1213) node{} -- (-5.6568,-5.6568);
		
		\end{tikzpicture}
		\caption{Three S-trees $S_1$,$S_2$ and $S_3$}
	\end{subfigure}
	\hspace{4em}
	\begin{subfigure}[c]{0.3\textwidth}
		\begin{tikzpicture}[thick,scale=0.2]%
		\draw 
		(0,0) node[fill=gray]{$v^{*}$}
		(0,5) node{} -- (24-24,0)
		(0,10) node{} -- (24-24,5)
		(3.5355,3.5355) node{} -- (0,0)
		(31.0711-24,2*3.5355) node{} --	(3.5355,3.5355)
		(27.5355-24,-3.5355) node{} -- (24-24,0)
		(31.0711-24,-7.0711) node{} -- (27.5355-24,-3.5355)
		(31.0711-24,0) node{} -- (27.5355-24,-3.5355)
		(24-24,-7.0711) node{} -- (27.5355-24,-3.5355)
		(34.6066-24,-3.5355) node{} -- (31.0711-24,-7.0711)
		(27.5355-24,-10.6066) node{} -- (31.0711-24,-7.0711)
		
		(20.4645-24,-3.5355) node{} -- (24-24,0)
		(16.9289-24,0) node{} -- (20.4645-24,-3.5355)
		(16.9289-24,-7.0711) node{} -- (20.4645-24,-3.5355);
		
		\end{tikzpicture}
		\vspace{.5em}
		\caption{$\bigoasterisk\limits_{i=1}^{3}(S_i,v_i)$}
	\end{subfigure}

	~ 
	\caption{coalescence}
\end{figure}
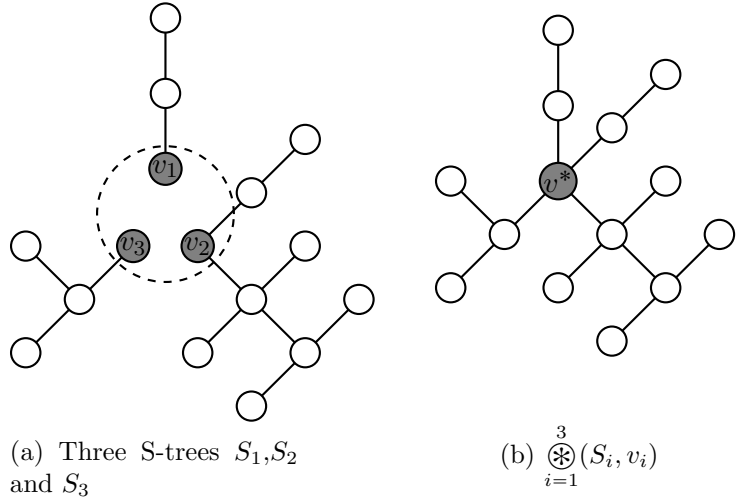



The following theorem prove that S-trees are closed under S-coalescence.
\begin{theorem}\label{TScoalescence}
Let \(S_{1}, \dots, S_{k}\) be \(k\) S-trees, and \(v_{i} \in \Supp{}{S_{i}}\). Then \(\bigoasterisk_{i=1}^{k} (S_{i},v_{i})\) is an S-tree.
\end{theorem}

\begin{proof}
Let \(T= \bigoasterisk_{i=1}^{k} (S_{i},v_{i}) \). By Lemma \ref{vecsupp}, for every \(i \in [k]\) there exists a vector \(x(i) \in \N{S_{i}}\) such that \(\Supp{S_{i}}{x(i)}=\Supp{}{S_{i}}\), and without loss of generality \(x(i)_{v_{i}}=1\). Now we define the vector \(y \in \mathbb{R}^{T}\) as:
\[
\down{y}{T}{$S_{i}-v_{i}$}=\down{x(i)}{$S_{i}$}{$S_{i}-v_{i}$}
\]
for each \(i \in [k]\), and \(y_{v^{*}}=1\), where \(v^{*}\) is the new vertex, which substitutes vertices \(v_{1}, \dots, v_{k}\). Let \(j \in V(S_{i}-v_{i})\). Then
\[
(A(T)y)_{j}=(A(S_{i})x(i))_{j}= 0
\]
and clearly,
\[
(A(T)y)_{v^{*}}=\sum_{i=1}^{k}(A(S_{i})x(i))_{v_{i}}=0
\]
Hence
\[
 \{v^{*}\} \cup \left( \bigcup_{i=1}^{k} \Supp{}{S_{i}}\setminus \{v_{i}\}\right)   \subset \Supp{}{T}
\]
Thus
\[
N[\Supp{}{T}]=V(T)
\]
This prove that \(T=\bigoasterisk_{i=1}^{k} (S_{i},v_{i})\) is an S-tree.
\end{proof}

\begin{corollary}
	Let \(S_{1}, \dots, S_{k}\) be \(k\) S-trees, and \(v_{i} \in \Supp{}{S_{i}}\) for \(i \in [k]\). Then
	\begin{enumerate}
		\item \(\Core{\bigoasterisk_{i=1}^{k} (S_{i},v_{i})}=\bigcup_{i=1}^{k}\Core{S_{i}}\).
		\item \(\core{\bigoasterisk_{i=1}^{k} (S_{i},v_{i})}= \sum_{i=1}^{k} \core{S_{i}}\).
		\item \(\Supp{}{\bigoasterisk_{i=1}^{k} (S_{i},v_{i})}= \{v^{*}\} \bigcup_{i=1}^{k}\left(\Supp{}{S_{i}} \setminus \{v_{i}\} \right)\).
		\item \(\supp{}{\bigoasterisk_{i=1}^{k} (S_{i},v_{i})}=1-k+\sum_{i=1}^{k} \supp{}{S_{i}}\).
		\item \(\rank{\bigoasterisk_{i=1}^{k} (S_{i},v_{i})}= \sum_{i=1}^{k} \rank{S_{i}}  \).
		\item \(\nulidad{\bigoasterisk_{i=1}^{k} (S_{i},v_{i})}=1-k+\sum_{i=1}^{k} \nulidad{S_{i}}\).
		\item \(\nu(\bigoasterisk_{i=1}^{k} (S_{i},v_{i}))= \sum_{i=1}^{k} \nu(S_{i}) \).
		\item \(m(\bigoasterisk_{i=1}^{k} (S_{i},v_{i})) < \prod_{i=1}^{k} m(S_{i})\).
		\item \(\alpha(\bigoasterisk_{i=1}^{k} (S_{i},v_{i}))=1-k+\sum_{i=1}^{k} \alpha(S_{i})\).
	\end{enumerate}
\end{corollary}

\begin{proof}
Clearly 1 and 4 follow from 3, and 2 follows from 1. On another hand, 5 follows from 2 and Theorem \ref{Trknulloftrees}. The statement 6 follows from 2, 4, and 5. The stament 7 follows from Theorem \ref{MatchingMaximumStrees} and 2. Finally 9 follows from Corollary \ref{nualpha_coresupp} and 4.
	
	In order to prove 3 let \(y \in \Supp{}{\bigoasterisk_{i=1}^{k} (S_{i},v_{i})}\), without loss of generality we assume that \(y_{v^{*}}=1\). As
	\[
	A(S_{i})\left( \up{(\down{y}{$\bigoasterisk_{i=1}^{k} (S_{i},v_{i})$}{$S_{i}-v_{i}$})}{$S_{i}$}{$S_{i}-v_{i}$} + e_{v_{i}} \right) =\theta
	\]
	where \(e_{v_{i}}\) is a canonical vector of \(\mathbb{R}^{S_{i}}\). Then
	\[
	 \Supp{}{\bigoasterisk_{i=1}^{k} (S_{i},v_{i})} \subset \{v^{*}\} \cup \left( \bigcup_{i=1}^{k} \Supp{}{S_{i}}\setminus \{v_{i}\}\right)   
	\]
	and hence, by the proof of Theorem \ref{TScoalescence}, 3 follows.
	
	To prove 8, just note that there is an injection between the maximum matchings of \(\bigoasterisk_{i=1}^{k} (S_{i},v_{i})\) and \(\prod_{i=1}^{k} M(S_{i}) \). Let \(M \in \mathcal{M}(\bigoasterisk_{i=1}^{k} (S_{i},v_{i}))\), let \(u_{i} \in V(S_{i})\) such that \(\{u_{i},v^{*}\}\in M\), then
	\[
	M-\{u_{i},v^{*}\}+\{u_{i},v_{i}\} \in \prod_{i=1}^{k} M(S_{i})
	\]
	But this injection is no onto. Let \(M(i) \in \mathcal{M}(S_{i})\) such that \(v_{i} \in V(M(i))\). No \(M \in \mathcal{M}(\bigoasterisk_{i=1}^{k} (S_{i},v_{i}))\) can reach \(\prod_{i=1}^{k} M(i)\).
\end{proof}
Let \(T\) be a tree, and \(u,v \in V(T)\), here and subsequently, \(T(u \rightarrow v)\) stands for the following subtree  of \(T\):
\[
T(u \rightarrow v) := T\left\langle \{ x \in V(T): v \in V(uP_{T}x) \} \right\rangle 
\]

%
\tikzstyle{every node}=[circle, draw, fill=white!,
inner sep=0.1pt, minimum width=11pt]

\begin{figure}[h]
	\centering
	\begin{tikzpicture}[thick,scale=0.2]%
	
	\draw
	(0,0) node{}
	(0,5) node[fill=gray]{$u$} -- (0,0)
	(0,10) node{} -- (0,5)
	(-5,10) node{} -- (0,10)
	(0,15) node{} -- (0,10)
	(-5,15) node{} -- (0,15)
	(0,20) node{} -- (0,15)
	(-5,20) node{} -- (0,20)
	(5,15) node{} -- (0,15)
	(10,15) node[fill=gray]{$v$} -- (5,15)
	(15,15) node{} -- (10,15)
	(20,15) node{} -- (15,15)
	(10,10) node{} -- (10,15)
	(15,10) node{} -- (15,15)
	(15,5) node{} -- (15,10);
	
	\draw [dotted] (10,17) -- (20,17);
	\draw [dotted] (8,15) -- (8,10);
	\draw [dotted] (17,11.2) -- (17,5);
	\draw [dotted] (20,13) -- (18.8,13);
	\draw [dotted] (10,8) -- (10.8,8);
	\draw [dotted] (13,5) -- (13,5.8);

	\draw [dotted] (20,13) arc (-90:90:2);
	\draw [dotted] (10,17) arc (90:180:2);
	\draw [dotted] (13,5) arc (180:360:2);
	\draw [dotted] (17,11) arc (180:90:2);
	\draw [dotted] (8,10) arc (180:270:2);
	\draw [dotted] (13,6) arc (0:90:2);
	
	\draw (15,22) node[white]{\textcolor{black}{$T(u \to v)$}};
	
	\end{tikzpicture}
	\caption{$T$ and $T(u \to v)$.}
\end{figure}
%


The set of all supported vertices with degree greater than 1 carries structural information about trees. 
\begin{definition}
Let \(T\) be a tree. Then
\[
\ISupp{T}:=\{v \in \Supp{}{T}\; :\; \deg(v)>1\}
\]
\end{definition}

\begin{theorem}
Let \(S\) be a S-tree such that \(\ISupp{S} \neq \emptyset \). Then \(S\) is an S-coalescence of S-trees.
\end{theorem}

\begin{proof}
Let \(v \in \ISupp{S} \). Consider the neighbors of \(v\):  \(N(v)= \lbrace u_{1},u_{2}, \ldots, u_{k} \rbrace \). From  \(T\) we obtain a  forest \(F\) with \(k\) trees. Vertices of \(F\) are 
\[
V(F) :=\lbrace v_{1}^{*},v_{2}^{*}, \ldots, v_{d}^{*} \rbrace \cup V(S) \setminus {v} 
\]
Edges of \(F\) are
\[
E(F) := \{\{u_{i},v_{i}^{*}\}: 1 \leq i \leq k \} \cup \left( E(S) \setminus \{ \{u_{i},v\} : 1 \leq i \leq k \} \right)
\]

Let call \(S(v_{i}^{*})\) to the tree of \(F\) that contains the vertex \(v_{i}^{*}\). Let \(y \in\N{S}\), and \(x\) a vector of \(\mathbb{R}^{S(v_{i}^{*})}\) such that \(\down{x}{$S(v_{i}^{*})$}{$S(v \rightarrow u_{i})$}= \down{y}{S}{$S(v \rightarrow u_{i})$}\), and \(x_{v_{i}^{*}}=y_{v}\). For \(j \in V(S(v_{i}^{*})) \setminus \{v_{i}^{*}\}\) is clear that for each \(i \in [k]\)
\[
\left( A(S(v_{i}^{*}))x\right)_{j}= \left( A(S)y\right)_{j}=0
\]
and that 
\[
\left(A(S(v_{i}^{*}))x\right)_{v_{i}^{*}}=0
\]
Hence \( \{v_{i}^{*}\} \cup \left( \Supp{}{S} \cap V(S(v_{i}^{*}) \right) \subset \Supp{}{S(v_{i}^{*}}) \). Therefore \(N[\Supp{}{S(v_{i}^{*}}] =V(S(v_{i}^{*}))\). Thus \(S(v_{i}^{*})\) is an S-tree. Clearly \(S=\bigoasterisk_{i=1}^{k} (S_{i},v_{i}^{*})\).
\end{proof}

We can apply the former decomposition \(|\ISupp{S}|\) times in order to get the set of stellare trees that form the given S-tree \(S\).



\section{S-atoms}

In \cite{molina_jaume2016tree} it was proved that the number of maximum matchings of a tree \(T\) depends on its S-forest:
\[
m(T)= \prod_{S \in \mathcal{F}_{S}(T)} m(S) 
\]
Now we will give a refinement of this result in terms of S-atoms.
\begin{definition}
	Let \(S\) be a S-tree, the \textbf{A-set} of \(S\), denoted by \(\mathcal{F}_{A}(S)\), is the set of all connected components that remains after taking away all the edges between core vertices. 
\end{definition}

%
\begin{figure}[h]
	\centering
	\begin{subfigure}[h]{0.3\textwidth}
		\begin{tikzpicture}[thick,scale=0.2]%
		
		\draw [ultra thick] (0,5) -- (0,0);
		\draw [ultra thick] (-10,10) -- (-5,10);
		\draw 
		(0,0) node[fill=gray]{}
		(0,5) node[fill=gray]{}
		(-5,0) node{} -- (0,0)
		(0,-5) node{} -- (0,0)
		(5,0) node{} -- (0,0)
		(5,5) node{} -- (0,5)
		(-5,5) node{} -- (0,5)
		(0,10) node{} -- (0,5)
		(-5,10) node[fill=gray]{} -- (0,10)
		(-5,15) node{} -- (-5,10)
		(-10,10) node[fill=gray]{}
		(-10,5) node{} -- (-10,10)
		(-10,15) node{} -- (-10,10);
		
		\end{tikzpicture}
		\caption{$S$}
	\end{subfigure}
	\hspace{4em}
	\begin{subfigure}[h]{0.3\textwidth}
		
		\begin{subfigure}[h]{0.3\textwidth}
			\begin{tikzpicture}[thick,scale=0.2]%
			\draw 
			(0,0) node[fill=gray]{}
			(0,5) node[fill=gray]{} 
			(-5,0) node{} -- (0,0)
			(0,-5) node{} -- (0,0)
			(5,0) node{} -- (0,0)
			(5,5) node{} -- (0,5)
			(-5,5) node{} -- (0,5)
			(0,10) node{} -- (0,5)
			(-5,10) node[fill=gray]{} -- (0,10)
			(-5,15) node{} -- (-5,10)
			(-10,10) node[fill=gray]{}
			(-10,5) node{} -- (-10,10)
			(-10,15) node{} -- (-10,10);
			
			\end{tikzpicture}
		\end{subfigure}
		\caption{$\mathcal{F}_A(S)$}
	\end{subfigure}
	~ 
	\caption{An S-tree $S$ and its  A-set $\mathcal{F}_A(S)$.}
\end{figure}
We usually think \(\mathcal{F}_A(S)\) as a forest, so \(V(\mathcal{F}_A(S))\) is the set of vertices of \(\mathcal{F}_A(S)\) looked as forest, etc. Even we say sometimes A-forest instead of A-set.

%
\begin{theorem} \label{ThS-atoms1}
	Let \(S\) be a S-tree. The A-set of \(S\) is a set of S-trees. Let \(\mathfrak{A} \in \mathcal{F}_{A}(S)\), then \(\Core{\mathfrak{A}}=\Core{S} \cap V(\mathfrak{A})\), and \(\Supp{}{\mathfrak{A}}=\Supp{}{S} \cap V(\mathfrak{A})\).
\end{theorem}

\begin{proof}
	Let \(\mathfrak{A} \in \mathcal{F}_{A}(S)\), and \(x  \in \N{S} \) such that \(\Supp{}{S}=\Supp{S}{x}\), see Lemma 7 in \cite{nylen1998null}. As \(x_{u}=0\) for all \(u \in \Core{S} \), it is easy to check that \(A(\mathfrak{A})\down{x}{T}{$\mathfrak{A}$}= \theta\). Hence \(\Supp{}{S} \cap V(\mathfrak{A}) \subset \Supp{}{\mathfrak{A}} \subset V(\mathfrak{A}) \). Let \(v \in \Core{S} \cap \mathfrak{A}\), then there exists \(u \in \Supp{}{S}\) such that \(u \sim v\). By definition of A-forest, \(u \in V(\mathfrak{A})\). Thus \(v \in N[\Supp{}{S} \cap V(\mathfrak{A})]\). Therefore \(V(\mathfrak{A}) \subset N[\Supp{}{S} \cap V(\mathfrak{A})] \). Hence \(N[\Supp{}{S} \cap V(\mathfrak{A})]=V(\mathfrak{A})\).
\end{proof}

\begin{corollary} \label{CoroS-Atoms1}
	Let \(S\) be and S-tree, then \(\Supp{}{S}=\bigcup_{\mathfrak{A} \in \mathcal{F}_{A}(S)} \Supp{}{\mathfrak{A}} \) and \(\Core{S}=\bigcup_{\mathfrak{A} \in \mathcal{F}_{A}(S)} \Core{\mathfrak{A}}  \).
\end{corollary}

\begin{proof}
	As \(V(S)=\Supp{}{S} \cup \Core{S}\) and these sets are disjoint, by Theorem \ref{ThS-atoms1}
	\begin{align*}
	v(S)  & = \supp{}{S} + \core{S}\\
	{} &  \geq \sum_{\mathfrak{A} \in \mathcal{F}_{A}(S)} \supp{}{\mathfrak{A}} + \sum_{\mathfrak{A} \in \mathcal{F}_{A}(S)} \core{\mathfrak{A}}\\
	{} & = \sum_{\mathfrak{A} \in \mathcal{F}_{A}(S)} v(\mathfrak{A})\\
	{} & = v(S)
	\end{align*}
	from where the corollary follows.
\end{proof}

\begin{definition}
	An S-tree \(\mathfrak{A}\) is an \textbf{S-atom} if \(\mathcal{F}_{A}(\mathfrak{A})=\mathfrak{A}\), i.e., \(\mathfrak{A}\) has no edges between core vertices.
\end{definition}

\begin{corollary}
		Let \(S\) be a S-tree, the A-set of \(S\) is a set of S-atoms.
\end{corollary}
%
%

Now we study some properties of S-atoms. The first one is a direct consequence of the definitions of S-coalescence and S-atom.

\begin{proposition}
	S-atoms are closed under S-coalescence.
\end{proposition}

The next proposition shows that in S-atoms the core and the support form a bipartition.

\begin{proposition}
	Let \(\mathfrak{A}\) be an S-atom. Then \(\mathfrak{A}\) is \((\Core{\mathfrak{A}},\Supp{}{\mathfrak{A}})\)-bipartite.
\end{proposition}

\begin{proof} The proposition follows from the next clear statements:
	\begin{enumerate}
		\item If \(u,v \in \Core{\mathfrak{A}}\), then \(d(u,v)\) is even.
		\item If \(u,v \in \Supp{}{\mathfrak{A}}\), then \(d(u,v)\) is even.
		\item If \(u \in \Core{\mathfrak{A}}\) and \(v \in \Supp{}{\mathfrak{A}}\), then \(d(u,v)\) is odd.
	\end{enumerate}
\end{proof}

The maximum degree of all the core vertices in an S-atom \(\mathfrak{A}\) will be denoted by \(\Delta_{\text{core}}(\mathfrak{A})\). It provides a lower bound for the nullity of \(\mathfrak{A}\). 

\begin{theorem}\label{SAtomTcotanulidad}
	Let \(\mathfrak{A}\) be an S-atom. Then \(\nulidad{\mathfrak{A}} \geq \Delta_{\text{core}}(\mathfrak{A})-1\)
\end{theorem}

\begin{proof}
	Let \(\mathfrak{A}\) an S-atom, and \(u \in \Core{\mathfrak{A}}\) such that \(\deg(u)=\Delta_{\text{core}}(\mathfrak{A})\). It is clear that:
	\begin{align*}
	\supp{}{\mathfrak{A}} & \geq \Delta_{\text{core}}(\mathfrak{A}) +
	\sum_{\substack{v \in  \Core{\mathfrak{A} }\\ d(v,u)=2}} 1 +
	\sum_{\substack{v \in  \Core{\mathfrak{A} }\\ d(v,u)=4}} 1 + \cdots +
	\sum_{\substack{v \in  \Core{\mathfrak{A} }\\ d(v,u)=diam(\mathfrak{A})-2}} 1\\
	{} & = \Delta_{\text{core}}(\mathfrak{A}) + (\core{\mathfrak{A}}-1)
	\end{align*}
	Then, by Theorem \ref{Trknulloftrees}, \(null(\mathfrak{A})=\supp{}{\mathfrak{A}}-\core{\mathfrak{A}} \geq \Delta_{\text{core}}(\mathfrak{A})-1\).
\end{proof}

By construction, for any S-tree \(S\) there is a set of edges in \(E(S)\) that are not edges of any S-atom of \(\mathcal{F}_{A}(S)\). This edges are called the \textbf{bond edges} of \(S\), and it is denoted by \(\BondE{T}\):
\[
\BondE{T}:=E(S) \setminus E(\mathcal{F}_{A}(S)) 
\]
Given \(\mathfrak{A}_{1}, \mathfrak{A}_{2} \in \mathcal{F}_{A}(S) \) we say that \(\mathfrak{A}_{1}\) and \(\mathfrak{A}_{2}\) are \textbf{adjacent atoms}, denoted by \(\mathfrak{A}_{1} \sim \mathfrak{A}_{2}\), if there exists \(e \in \BondE{T}\) such that  \(e \cap V(\mathfrak{A}_{1}) \neq \emptyset \) and  \(e \cap V(\mathfrak{A}_{1}) \neq \emptyset \).

\begin{proposition}
	Let  \(S\) be an S-tree, and \(M \in \mathcal{M}(T)\). Then \(M \cap \BondE{S} = \emptyset \).
\end{proposition}

\begin{proof}
	It follows from Corollary \ref{CoroMatching}.
\end{proof}

The set of bond edges of an arbitrary tree \(T\) is defined by
\[
\BondE{T} := \bigcup_{S \in \mathcal{F}_{S}(T)} \BondE{S}
\]

%

\begin{theorem}
	Let \(S\) be a S-tree. Then
	\[
	m(S)= \prod_{\mathfrak{A} \in \mathcal{F}_{A}(T)} m(\mathfrak{A}) 
	\]
\end{theorem}

\begin{proof}
	By Theorem \ref{MatchingMaximumStrees}, a maximum matching \(M\) on \(S\) does not have any edge between two core vertices. By Corollary \ref{CoroS-Atoms1} the next two statements hold:
	\begin{itemize}
		\item If \(M \in \mathcal{M}(S)\), then for all \(\mathfrak{A} \in \mathcal{F}_{A}(S)\) holds \(M \cap E(\mathfrak{A}) \in \mathcal{M}(\mathfrak{A})\).
		\item Let  \(M_{\mathfrak{A}} \in \mathcal{M}(\mathfrak{A})\) for \(\mathfrak{A} \in \mathcal{F}_{A}(S)\), then \(M=\bigcup_{\mathfrak{A} \in \mathcal{F}_{A}(S)} M_{\mathfrak{A}} \in \mathcal{M}(S)\).
	\end{itemize}
This implies that \(	m(S)= \prod_{\mathfrak{A} \in \mathcal{F}_{A}(T)} m(\mathfrak{A}) \).
\end{proof}
\begin{corollary}
	Let \(T\) be a tree. Then
	\[
	m(T)= \prod_{\mathfrak{A} \in \mathcal{F}_{A}(T)} m(\mathfrak{A}) 
	\]
\end{corollary}

\begin{theorem}
	Let \(S\) be an S-tree. Then 
	\begin{enumerate}
		\item \(\nulidad{S} =\sum_{\mathfrak{A} \in \mathcal{F}_{A}(S)} \nulidad{\mathfrak{A}} \).
		\item \(\rank{S} =\sum_{\mathfrak{A} \in \mathcal{F}_{A}(S)} \rank{\mathfrak{A}} \).
	\end{enumerate}
\end{theorem}

\begin{proof}
	By Theorem \ref{Trknulloftrees} and Corollary \ref{CoroS-Atoms1}
	\begin{align*}
	\nulidad{S} & = \supp{}{S}-\core{S}\\
	{} &  = \sum_{\mathfrak{A} \in \mathcal{F}_{A}(S)} \supp{}{\mathfrak{A}} - \sum_{\mathfrak{A} \in \mathcal{F}_{A}(S)} \core{\mathfrak{A}}\\
	{} & = \sum_{\mathfrak{A} \in \mathcal{F}_{A}(S)} \nulidad{\mathfrak{A}}
	\end{align*}
	and, as in S-trees in general and S-atoms in particular holds that the rank is two times the cardinality of the core
	\begin{align*}
	\rank{S} & = v(S)-\nulidad{S}\\
	{} & = (\supp{}{S}+\core{S})-(\supp{}{S}-\core{S})\\
	{} & = 2 \core{S}\\
	{} & = 2 \sum_{\mathfrak{A} \in \mathcal{F}_{A}(S)} \core{\mathfrak{A}}\\
	{} & = \sum_{\mathfrak{A} \in \mathcal{F}_{A}(S)} \rank{\mathfrak{A}} 
	\end{align*}
\end{proof}

\begin{corollary}
Let \(S\) be an S-tree. Then \(\N{S} = \displaystyle \bigoplus_{\mathfrak{A} \in \mathcal{F}_{A}(S)} \up{\N{\mathfrak{A}}}{$S$}{$\mathfrak{A}$}\).
\end{corollary}

\begin{proof}
	Let \(x \in \N{S}\), given \(\mathfrak{A} \in \mathcal{F}_{A}(S)\) is clear that \(A(\mathfrak{A})\down{x}{$S$}{$\mathfrak{A}$} =\theta \in \mathbb{R}^{\mathfrak{A}} \). Hence
	\[
	\N{S} \subset \bigoplus_{\mathfrak{A} \in \mathcal{F}_{A}(S)} \up{\N{\mathfrak{A}}}{$S$}{$\mathfrak{A}$}
	\]
\end{proof}

For any S-atom \(\mathfrak{A}\), its rank is \(2\core{\mathfrak{A}}\). But actually, more can be say.

\begin{theorem}
	Let \(\mathfrak{A}\) be an S-atom. The set 
	\[
	B_{\mathcal{R}}(\mathfrak{A}):=\{e_{v},e_{R(v)}\in \mathbb{R}^{\mathfrak{A}}: v \in \Core{\mathfrak{A}}\}
	\]
	is a basis of \(\Rank{\mathfrak{A}}\).
\end{theorem}
\begin{proof}
	Let \(A_{u}\) be the column of \(A(\mathfrak{A})\), the adjacency matrix of \(\mathfrak{A}\), associated to the vertex \(u\). If \(u \in \Supp{}{\mathfrak{A}}\), then 
	\[
	A_{u}=\sum_{\substack{v \in \Core{\mathfrak{A}} \\ v \sim u} } e_{v}
	\]
	If \(u \in \Core{\mathfrak{A}}\), then 
	\[
	A_{u} = \sum_{\substack{v \in \Supp{}{\mathfrak{A}} \\ v \sim u} } e_{v}=e_{R(u)}
	\]
	Hence \(\Rank{\mathfrak{A}} \subset \langle B_{\mathcal{R}}(\mathfrak{A}) \rangle \).
	Clearly \(B_{\mathcal{R}}(\mathfrak{A})\) is a set of linear independent vectors, and \(|B_{\mathcal{R}}(\mathfrak{A})|=2\core{\mathfrak{A}}\). Hence \(B_{\mathcal{R}}(\mathfrak{A})\) is a base of \(\Rank{\mathfrak{A}}\).	
\end{proof}

\begin{corollary}
	Let \(S\) be an S-tree. Then \( \Rank{S} =  \displaystyle \bigoplus_{\mathfrak{A} \in \mathcal{F}_{A}(S)} \up{\Rank{\mathfrak{A}}}{$S$}{$\mathfrak{A}$}\).
\end{corollary}

\begin{proof}
	Let \(A_{u}\) be the column of \(A(S)\) associated to the vertex \(u\). If \(u \in \Supp{}{S}\), then there exists only one \(\mathfrak{A} \in \mathcal{F}_{A}(S)\) such that \(u \in \Supp{}{\mathfrak{A}}\). Thus
	\[
	A_{u}=\sum_{\substack{v \in \Core{\mathfrak{A}} \\ v \sim u} } \up{e_{v}}{S}{$\mathfrak{A}$}
	\]
	If \(u \in \Core{S}\), then there exists only a \(\mathfrak{A}  \in \mathcal{F}_{A}(S)\) such that \(u \in \Core{\mathfrak{A}}\). Thus 
	\[
	A_{u} = \up{e_{R(u)}}{$S$}{$\mathfrak{A}$}+\sum_{\{u,v\} \in \BondE{S}} \up{e_{v}}{$S$}{$\mathfrak{A}$}
	\]
\end{proof}

\begin{corollary}
	Let \(S\) be an S-tree. Then
	\[
	\bigcup_{\mathfrak{A} \in \mathcal{F}_{A}(S)}
	\up{ B_{\mathcal{R}}(\mathfrak{A})}{S}{$\mathfrak{A}$}
	\]
	is a basis of \(\Rank{S}\).
\end{corollary}

With \(B_{\mathcal{C}}(\mathbb{R}^{G})\) we denote the standard basis of \(\mathbb{R}^{G}\).

\begin{corollary}
	Let \(T\) be a tree. Then
	\[
	\bigcup_{N \in \mathcal{F}_{N}(T)} \up{B_{\mathcal{C}}(\mathbb{R}^{G})}{T}{N}
	\cup
	\bigcup_{S \in \mathcal{F}_{S}(T)}
	\up{
		\left(	
		\bigcup_{\mathfrak{A} \in \mathcal{F}_{A}(S)}
		\up{ B_{\mathcal{R}}(\mathfrak{A})}{S}{$\mathfrak{A}$}
		\right)
	}{T}{S}
	\]
	is a basis of \(\Rank{T}\).
\end{corollary}

Thus
\[
\{e_{1},e_{\{2,3\}},e_{4}, e_{\{6,7\}},e_{5}, e_{\{7,8\}},e_{9},e_{\{10,11,12\}},e_{13}, \dots, e_{18}\}
\]
is a basis of \(\Rank{T}\), for the tree \(T\) in Figure 2.

We now define the A-set for arbitrary trees.
\begin{definition}
	Let \(T\) be a tree, the \textbf{A-set} of \(T\) is
	\[
	\mathcal{F}_{A}(T):=\bigcup_{S \in \mathcal{F}_{S}(T)} \mathcal{F}_{A}(S)
	\]
\end{definition}

From the definitions of S-set and S-atom we have that the A-set of a tree \(T\) is a set of S-atoms.



\begin{corollary}
	Let \(T\) be a tree. Then
	\begin{align*}
	\Rank{T}& = \bigoplus_{N \in \mathcal{F}_{N}(T)} 
	\up{\Rank{N}}{$S$}{$N$}
	\oplus
	\bigoplus_{\mathfrak{A} \in \mathcal{F}_{A}(T)}
	\up{\Rank{\mathfrak{A}}}{$S$}{$\mathfrak{A}$}
	\\
	\rank{T} &=\sum_{N \in \mathcal{F}_{N}(T)} \rank{N}+\sum_{\mathfrak{A} \in \mathcal{F}_{A}(T)} \rank{\mathfrak{A}}\\
	\N{T} &= \bigoplus_{\mathfrak{A} \in \mathcal{F}_{A}(T)} 
	\up{\N{\mathfrak{A}}}{$S$}{$\mathfrak{A}$}
	\\
	\nulidad{T} & = \sum_{\mathfrak{A} \in \mathcal{F}_{A}(T)} \nulidad{\mathfrak{A}}
	\end{align*}
\end{corollary}

\section{S-basic trees}

S-atoms whose core vertices have degree 2 play a key role in order to find \(\{-1,0,1\}\)-bases of the null space of a tree. 
\begin{theorem}
 Let \(\mathfrak{B}\) be an S-atom. The following are equivalent:
 \begin{enumerate}
 	\item \(\Delta_{\text{core}}(\mathfrak{B})=2\).
 	\item \(\supp{}{\mathfrak{B}}=\core{\mathfrak{B}}+1\).
 	\item \(\nulidad{\mathfrak{B}}=1\).
 	\item \(\nu(\mathfrak{B})=\dfrac{n-1}{2} \).
 	\item \(\alpha(\mathfrak{B})=\dfrac{n+1}{2}\). 
 \end{enumerate}
An S-atom that satisfies any (and hence all) of these conditions is called an \textbf{S-basic tree} (or just S-basic).
\end{theorem}

\begin{proof} \label{SBasicP1}
	\((1 \Rightarrow 2) \) As \(\mathfrak{B}\) is an S-atom and all its core-vertices have degree 2
	\begin{align*}
	2\core{\mathfrak{B}} = & e(\mathfrak{B})\\
	{}= & v(\mathfrak{B})-1\\
	{}= & \core{\mathfrak{B}} + \supp{}{\mathfrak{B}}-1
	\end{align*}
	Then \(\core{\mathfrak{B}}=\supp{}{\mathfrak{B}}-1\).
	
	\((2 \Rightarrow 3 )\)  By Theorem \ref{Trknulloftrees} and (2)  \(\nulidad{\mathfrak{B}}=\supp{}{\mathfrak{B}}-\core{\mathfrak{B}}=1\).
	
	\((3 \Rightarrow 1 )\) If \(\mathfrak{B}\) is and S-tree and it has some core-vertex \(v\) of degree three or more, then, by Theorem \ref{SAtomTcotanulidad}, \(\nulidad{\mathfrak{B}} \geq 2\).
	
	\((2 \Rightarrow 4) \) As \(\mathfrak{B}\) is an S-atom \(n=\core{\mathfrak{B}}+\supp{}{\mathfrak{B}}\), by 2 \(n=2core{\mathfrak{B}}+1\).
	
	\((4 \Leftrightarrow 5) \) As \(\mathfrak{B}\) is an S-tree, \(\alpha(\mathfrak{B})+\nu(\mathfrak{B})=n \).	
	
	\( 4 \text{ and } 5 \Rightarrow 3\) By Corollary \ref{jaumelario}.
\end{proof}
\begin{figure}[h] 
	\centering
	\begin{tikzpicture}[thick,scale=0.2]
	
	\draw 
	(0,0) node[fill=gray]{}
	(5,10) node{} -- (5,5)
	(-5,0) node{} -- (0,0)
	(5,0) node{} -- (0,0)
	(5,5) node[fill=gray]{} -- (5,0)
	(10,0) node[fill=gray]{} -- (5,0)
	(15,0) node{} -- (10,0)
	(20,0) node[fill=gray]{} -- (15,0)
	(25,0) node{} -- (20,0)
	(15,-5) node[fill=gray]{} -- (15,0)
	(10,-5) node{} -- (15,-5)
	(5,-5) node[fill=gray]{} -- (10,-5)
	(0,-5) node{} -- (10,-5)
	(10,-10) node[fill=gray]{} -- (10,-5)
	(10,-15) node{} -- (10,-10);
	
	\end{tikzpicture}
	\caption{An $S$-basic tree $\mathfrak{B}$.}
\end{figure}
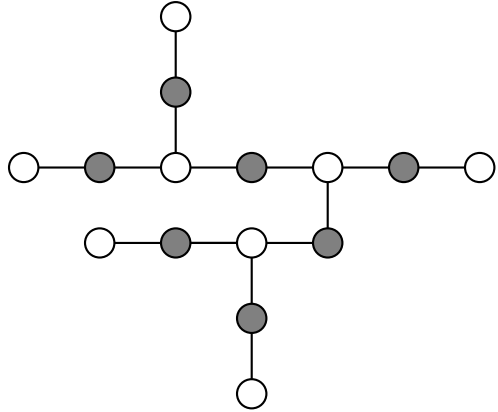
The following Corollary characterizes trees with nullity one.
\begin{corollary}
	A tree \(T\) has nullity one if and only if \(\mathcal{F}_{A}(T)=\{\mathfrak{B}\}\), where \(\mathfrak{B}\) is an S-basic.
\end{corollary}

\begin{theorem}
Let \(\mathfrak{B}\) be a subtree of an S-atom \(\mathfrak{A}\),  with \(v(\mathfrak{A})\geq 3\), such that:
	\begin{enumerate}
		\item if \(v \in V(\mathfrak{B}) \cap \Core{\mathfrak{A}}\), then \(deg_{\mathfrak{B}}(v)=2\),
		\item if \(v \in V(\mathfrak{B}) \cap \Supp{}{\mathfrak{A}}\), then \(N_{\mathfrak{B}}(v) = N_{\mathfrak{A}}(v)\).
	\end{enumerate}
Then \(\mathfrak{B}\) is an S-basic subtree of \(\mathfrak{A}\).
\end{theorem}

\begin{proof}
	Note that \(V(\mathfrak{B})= (V(\mathfrak{B}) \cap \Core{\mathfrak{A}}) \cup (V(\mathfrak{B}) \cap \Supp{}{\mathfrak{A}})\). Then the hypothesis implies that
	\begin{enumerate}
		\item \(v(\mathfrak{B}) \geq 3\).
		\item \( V(\mathfrak{B}) \cap \Core{\mathfrak{A}} \neq \emptyset\).
		\item \(V(\mathfrak{B}) \cap \Supp{}{\mathfrak{A}} \neq \emptyset\).
	\end{enumerate}
	Let \(h\) be a pendant vertex of \(\mathfrak{B}\). Consider the vector \(\overrightarrow{b}(h)\):
	\[
	\overrightarrow{b}(h)_{v}:=
	\left\lbrace 
	\begin{array}{ll}
	(-1)^{\frac{d(v,h)}{2}} & \text { if } d(i,j) \text{ is even}.\\
	0 & \text{ otherwise.}
	\end{array}
	\right. 
	\]
	Claim 1 \(A(\mathfrak{B})\overrightarrow{b}(h)=\theta\).
	
	Proof of Claim 1: As \(deg_{\mathfrak{B}}(h)=1\), \(h\) is a supported vertex of \(\mathfrak{A}\): \(h \in V(\mathfrak{B} \cap \Supp{}{\mathfrak{A}}\). Therefore, as \(\mathfrak{A}\) is an S-tree, for all \(u \in V(\mathfrak{B})\cap \Supp{}{\mathfrak{A}}\), \(d(h,u)\) is even, and for all \(u \in V(\mathfrak{B})\cap \Core{\mathfrak{A}}\), \(d(h,u)\) is odd. Then for \(u \in V(\mathfrak{B})\cap \Supp{}{\mathfrak{A}}\)
	\begin{equation*}
	\left(
	A(\mathfrak{B})\overrightarrow{b}(h)
	\right)_{u} = \sum_{v \sim u}\overrightarrow{b}(h)_{v}=0
	\end{equation*}
	because all neighbors \(v\) of \(u\) are in \(V(\mathfrak{B})\cap \Core{\mathfrak{A}}\). And for \(u \in V(\mathfrak{B})\cap \Core{\mathfrak{A}}\)
	\begin{equation*}
	\left(
	A(\mathfrak{B})\overrightarrow{b}(h)
	\right)_{u} = \overrightarrow{b}(h)_{v_{1}}+\overrightarrow{b}(h)_{v_{2}}=\pm 1 \mp 1 = 0
	\end{equation*}
	because all \(v\) neighbor of \(u\) is in \(V(\mathfrak{B})\cap \Supp{}{\mathfrak{A}}\).
	
	Claim 2: If \(v \in V(\mathfrak{B}) \cap \Supp{}{\mathfrak{A}}\), then \(\left(\overrightarrow{b}(h)\right)_{v} \neq 0\). Hence \(V(\mathfrak{B}) \cap \Supp{}{\mathfrak{A}} \subset \Supp{}{\mathfrak{B}}\), and then \(\Core{\mathfrak{B}} \subset V(\mathfrak{B}) \cap \Core{\mathfrak{A}}\).
	
	Claim 3: If \(v \in V(\mathfrak{B}) \cap \Core{\mathfrak{A}}\), then \(\left(\overrightarrow{b}(h)\right)_{v}=0\).
%
%
%
	
	Claim 4: \(V(\mathfrak{B})=N_{\mathfrak{B}}[V(\mathfrak{B}) \cap \Supp{}{\mathfrak{A}}]\).
	
	Proof of Claim 4: Let \(u \in V(\mathfrak{B})\), but \(u \notin \Supp{}{\mathfrak{A}}\). Then \(u \in \Core{\mathfrak{A}}\). Thus \(\deg_{\mathfrak{B}(u)=2}\). Therefore there exists \(w \in V(\mathfrak{B} \cap \Supp{}{\mathfrak{A}}\) such that \(w \sim u\). Hence \(u \in N_{\mathfrak{B}}[V(\mathfrak{B}) \cap \Supp{}{\mathfrak{A}}]\). Thus \(V(\mathfrak{B}) \subset N_{\mathfrak{B}}[V(\mathfrak{B}) \cap \Supp{}{\mathfrak{A}}] \subset V(\mathfrak{B}) \).
	
	From the Claim 4 we conclude that \(V(\mathfrak{B})=N_{\mathfrak{B}}[\Supp{}{\mathfrak{B}}]\), thus \(\mathfrak{B}\) is an S-tree. By Claim 2 and the hypothesis \(\Delta_{\text{core}}(\mathfrak{B})=2\). Therefore \(\mathfrak{B}\) is an S-basic. 
\end{proof}

%
%

\begin{proposition}
	Every S-atom \(\mathfrak{A}\) has at least an S-basic subtree.
\end{proposition}

\begin{proof}
The next algorithm builds a S-basic subtree from a given S-atom \(\mathfrak{A}\).
\begin{algorithm}[S-Basic Subtree Algorithm] \label{SBSA}
INPUT: An S-atom \(\mathfrak{A}\), and a vertex \(v\) of \(\mathfrak{A}\).
\begin{enumerate}
	\item \(i=0\).
	\item \(\mathcal{B}_{i}=\{v\}\).
	\item IF \(v \in \Core{\mathfrak{A}}\):
		\begin{enumerate}
			\item CHOSE \(u,w \in N_{\mathfrak{A}}(v)\).
			\item \(\mathcal{B}_{i}=\{u,w\}\).
		\end{enumerate}
	\item WHILE exist \(u \in \Core{\mathfrak{A}} \setminus B_{i} \) such that \(N_{\mathfrak{A}}(u) \cap B_{i} \neq \emptyset\):
		\begin{enumerate}
			\item \(i=i+1\).
			\item CHOSE \(v \in N_{\mathfrak{A}}(u) \setminus B_{i})\).
			\item \(B_{i}=B_{i-1}\cup \{u,v\}\).
		\end{enumerate}
	\item Output: \(\mathfrak{A} \langle B_{i} \rangle\). 
\end{enumerate}
\end{algorithm}
It is clear that the result of the algorithm is an S-basic subtree of \(\mathfrak{A}\). We write \(SB(\mathfrak{A},v)\) for a S-basic subtree of \(\mathfrak{A}\) obtained when we apply the S-Basic subtree algorithm to the S-atom \(\mathfrak{A}\) at vertex \(v\), see Figure 8.
\end{proof}

%
\tikzstyle{every node}=[circle, draw, fill=white!,
inner sep=0.1pt, minimum width=11pt]

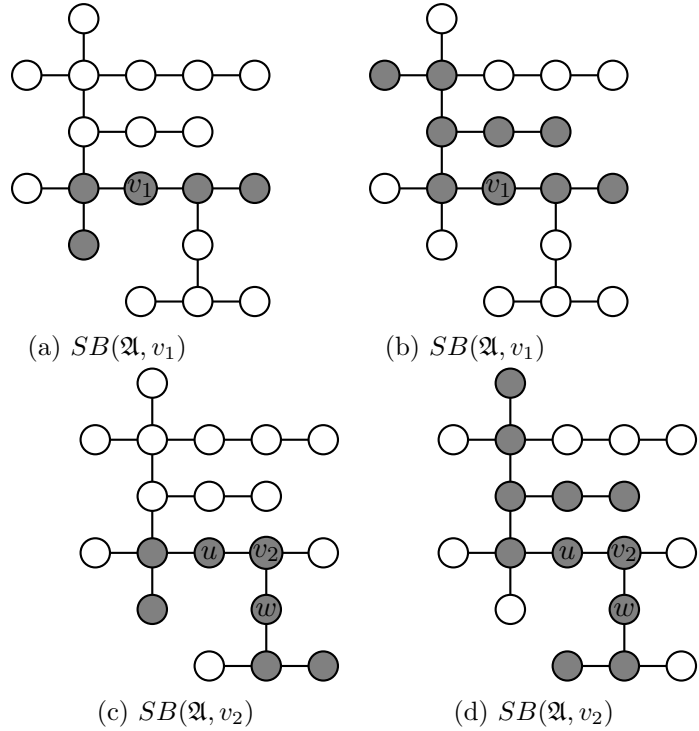
\begin{figure}[h]
	\centering
	\begin{subfigure}[t]{0.2\textwidth}
		\begin{tikzpicture}[thick,scale=0.15]%
		
		\draw
		(0,0) node[fill=gray]{}
		(0,5) node[fill=gray]{} -- (0,0)
		(0,10) node{} -- (0,5)
		(-5,5) node{} -- (0,5)
		(0,15) node{} -- (0,10)
		(-5,15) node{} -- (0,15)
		(0,20) node{} -- (0,15)
		(5,15) node{} -- (0,15)
		(10,15) node{} -- (5,15)
		(15,15) node{} -- (10,15)
		(10,10) node{} -- (5,10)
		(5,10) node{} -- (0,10)
		(5,5) node[fill=gray]{$v_1$} -- (0,5)
		(10,5) node[fill=gray]{} -- (5,5)
		(15,5) node[fill=gray]{} -- (10,5)
		(10,-5) node{} -- (10,0)
		(15,-5) node{} -- (10,-5)
		(5,-5) node{} -- (10,-5)
		(10,0) node{} -- (10,5);
		
		\end{tikzpicture}
		\caption{\(SB(\mathfrak{A},v_1)\)}
	\end{subfigure}
	\hspace{5em}
	\begin{subfigure}[t]{0.2\textwidth}
		\begin{tikzpicture}[thick,scale=0.15]%
		
		\draw
		(0,0) node{}
		(0,5) node[fill=gray]{} -- (0,0)
		(0,10) node[fill=gray]{} -- (0,5)
		(-5,5) node{} -- (0,5)
		(0,15) node[fill=gray]{} -- (0,10)
		(-5,15) node[fill=gray]{} -- (0,15)
		(0,20) node{} -- (0,15)
		(5,15) node{} -- (0,15)
		(10,15) node{} -- (5,15)
		(15,15) node{} -- (10,15)
		(10,10) node[fill=gray]{} -- (5,10)
		(5,10) node[fill=gray]{} -- (0,10)
		(5,5) node[fill=gray]{$v_1$} -- (0,5)
		(10,5) node[fill=gray]{} -- (5,5)
		(15,5) node[fill=gray]{} -- (10,5)
		(10,-5) node{} -- (10,0)
		(15,-5) node{} -- (10,-5)
		(5,-5) node{} -- (10,-5)
		(10,0) node{} -- (10,5);

		\end{tikzpicture}
		\caption{\(SB(\mathfrak{A},v_1)\)}
	\end{subfigure}
	\newline
	\begin{subfigure}[t]{0.2\textwidth}
		\begin{tikzpicture}[thick,scale=0.15]%
		
		\draw
		(0,0) node[fill=gray]{}
		(0,5) node[fill=gray]{} -- (0,0)
		(0,10) node{} -- (0,5)
		(-5,5) node{} -- (0,5)
		(0,15) node{} -- (0,10)
		(-5,15) node{} -- (0,15)
		(0,20) node{} -- (0,15)
		(5,15) node{} -- (0,15)
		(10,15) node{} -- (5,15)
		(15,15) node{} -- (10,15)
		(10,10) node{} -- (5,10)
		(5,10) node{} -- (0,10)
		(5,5) node[fill=gray]{$u$} -- (0,5)
		(10,5) node[fill=gray]{$v_2$} -- (5,5)
		(15,5) node{} -- (10,5)
		(10,-5) node[fill=gray]{} -- (10,0)
		(15,-5) node[fill=gray]{} -- (10,-5)
		(5,-5) node{} -- (10,-5)
		(10,0) node[fill=gray]{$w$} -- (10,5);

		\end{tikzpicture}
		\caption{\(SB(\mathfrak{A},v_2)\)}
	\end{subfigure}
	\hspace{5em}
	\begin{subfigure}[t]{0.2\textwidth}
		\begin{tikzpicture}[thick,scale=0.15]%
		
		\draw
		(0,0) node{}
		(0,5) node[fill=gray]{} -- (0,0)
		(0,10) node[fill=gray]{} -- (0,5)
		(-5,5) node{} -- (0,5)
		(0,15) node[fill=gray]{} -- (0,10)
		(-5,15) node{} -- (0,15)
		(0,20) node[fill=gray]{} -- (0,15)
		(5,15) node{} -- (0,15)
		(10,15) node{} -- (5,15)
		(15,15) node{} -- (10,15)
		(10,10) node[fill=gray]{} -- (5,10)
		(5,10) node[fill=gray]{} -- (0,10)
		(5,5) node[fill=gray]{$u$} -- (0,5)
		(10,5) node[fill=gray]{$v_2$} -- (5,5)
		(15,5) node{} -- (10,5)
		(10,-5) node[fill=gray]{} -- (10,0)
		(15,-5) node{} -- (10,-5)
		(5,-5) node[fill=gray]{} -- (10,-5)
		(10,0) node[fill=gray]{$w$} -- (10,5);

		\end{tikzpicture}
		\caption{\(SB(\mathfrak{A},v_2)\)}
	\end{subfigure}\label{Fig_S-basic_Subtree_Algrotihm}
	\caption{Some possible results of S-Basic Subtree Algorithm for the tree \(T\).}
\end{figure}
%

%
Let  \(\mathfrak{A}\) be an S-atom of order \(n\), and  let \(\mathfrak{B}\) be  an S-basic subtree of \(\mathfrak{A}\) of order \(m\). Let \(h\) be a pendant vertex of \(\mathfrak{B}\), the \textbf{basic vector} associated to \(\mathfrak{B}\) and \(h\), denoted by \(\overrightarrow{\mathfrak{B}(h)}\), is a vector in \( \mathbb{R}^{\mathfrak{B}}\), whose coordinates are indexed as follow: for each \(v \in V(\mathfrak{B})\)
\[
\overrightarrow{\mathfrak{B}(h)}_{v}:=
\left\lbrace 
\begin{array}{ll}
(-1)^{\frac{d(v,h)}{2}} & \text { if } d(i,j) \text{ is even}.\\
0 & \text{ otherwise.}
\end{array}
\right. 
\]
As the actual pendant vertex chosen is immaterial: let  \(u,v \in \Supp{}{\mathfrak{A}}\), then \(\overrightarrow{\mathfrak{B}(u)}=\pm \overrightarrow{\mathfrak{B}(v)}\). Thus we just write \(\overrightarrow{\mathfrak{B}}\).

\section{Null space of a tree}

%
%

%
%

Finding bases of the null space of a tree is a more involved task. We give an algorithm in order to obtain an appropriated basis for the null space of an S-atom.


Let \(U\) and \(V\) be two disjoin subtrees of a tree \(T\). The \textbf{in} of \(U\) from \(V\), denoted by \( \entrada{U}{V} \), is the vertex of \(U\) closest to \(V\). The \textbf{out} of \(U\) toward \(V\), denoted by \(\salida{U}{V}\), is the first vertex out of \(U\) toward \(V\).


%
\tikzstyle{every path}=[%
draw,%
thick,
rounded corners	=10pt,
color			=black,
line width		=1pt%
]
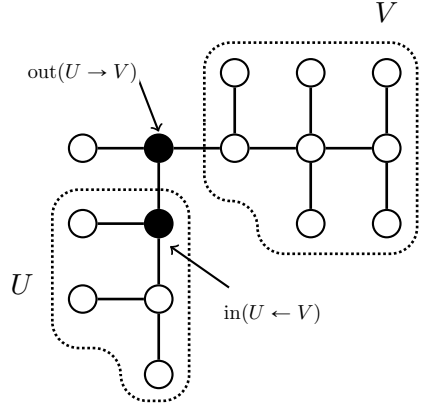
\begin{figure}[h]
	\centering
	\begin{tikzpicture}%
	\begin{scope}[%
	shift	={(0, -1)},
	rotate	={0},
	scale	={1.0}]
	\tikzstyle{every node}=[%
	draw,%
	shape		=circle,%
	color		=black,%
	thick
	]
	\node[name=01,	at={(-1, 0)},	node contents={}];
	\node[name=02,	at={(0, 0)},	fill=black,	node contents={}];
	\node[name=03,	at={(-1, -1)},	node contents={}];
	\node[name=04,	at={(0, -1)},	node contents={}];
	\node[name=05,	at={(0, -2)},	node contents={}];
	\path[draw] (02) to (01);
	\path[draw] (02) to (04);
	\path[draw] (04) to (03);
	\path[draw] (04) to (05);
	\end{scope}
	\begin{scope}[%
	shift	={(0.4, -0.55)},
	rotate	={0},
	xscale	=1.8,
	yscale	=1.4]
	\path[densely dotted] (0, 0) -- (0, -2) -- (-0.5, -2) -- (-0.5, -1.5) -- (-1, -1.5) -- (-1, -1) -- (-1, 0) -- cycle;
	\end{scope}
	\begin{scope}[%
	shift	={(0, 0)},
	rotate	={0},
	scale	={1.0}]
	\tikzstyle{every node}=[%
	draw,%
	shape		=circle,%
	color		=black,%
	thick
	]
	\node[name=06,	at={(0, 0)},	fill=black,	node contents={}];
	\node[name=07,	at={(-1, 0)},	node contents={}];
	\path[draw] (06) to (07);
	\end{scope}
	\begin{scope}[%
	shift	={(1, 0)},
	rotate	={0},
	scale	={1.0}]
	\tikzstyle{every node}=[%
	draw,%
	shape		=circle,%
	color		=black,%
	thick
	]
	\node[name=08,	at={(0, 0)},	node contents={}];
	\node[name=09,	at={(0, 1)},	node contents={}];
	\node[name=10,	at={(1, 0)},	node contents={}];
	\node[name=11,	at={(1, 1)},	node contents={}];
	\node[name=12,	at={(1, -1)},	node contents={}];
	\node[name=13,	at={(2, 0)},	node contents={}];
	\node[name=14,	at={(2, 1)},	node contents={}];
	\node[name=15,	at={(2, -1)},	node contents={}];
	\path[draw] (08) to (09);
	\path[draw] (08) to (10);
	\path[draw] (10) to (11);
	\path[draw] (10) to (12);
	\path[draw] (10) to (13);
	\path[draw] (13) to (14);
	\path[draw] (13) to (15);
	\end{scope}
	\begin{scope}[%
	shift	={(0.6, 0)},
	rotate	={0},
	xscale	=1.4,
	yscale	=1.4]
	\path[densely dotted] (0, 0) -- (0, 1) -- (1, 1) -- (2, 1) -- (2, -1) -- (0.5, -1) -- (0.5, -0.5) -- (0, -0.5) -- cycle;
	\end{scope}
	\path[draw] (02) to (06);
	\path[draw] (08) to (06);
	\draw[->, line width=0.8pt] (-0.3,1) to (0, 0.22);
	\draw[->, line width=0.8pt] (1.0, -1.9) to (0.15, -1.3);
	\draw (-1.8, -1.8) node[white]{\textcolor{black}{$U$}};
	\draw (3, 1.8) node[white]{\textcolor{black}{$V$}};
	\draw (1.5, -2.2) node[scale=0.7,white]{\textcolor{black}{in$(U \leftarrow V)$}};
	\draw (-1.0, 1)node[scale=0.7,white]{\textcolor{black}{out$(U \rightarrow V)$}};
	\end{tikzpicture}
	\caption{The in and out of a subgraph \( U \) in respect to the subgraph \( V \).}
	\label{def_in_out}
\end{figure}

\begin{algorithm} [S-Basis Forest Algorithm]
	INPUT: An S-atom \(\mathfrak{A}\).
	\begin{enumerate}
		\item Forest \(\mathcal{F}=\{\mathfrak{A}\}\).
		\item \(i=1\).
		\item Basic Forest \(\mathcal{F}_{B}(\mathfrak{A})=\emptyset\).
		\item Used vertices \(U=\emptyset\).
		\item \(MC=\emptyset\) an empty matrix whose column are indexed by the vertices of \(\mathfrak{A}\), and the row will be indexed by a set of S-basic trees that will be generated.
		\item WHILE \(\mathcal{F} \neq \emptyset \) then:
		\begin{enumerate}
			\item CHOSE \(H \in \mathcal{F}\)
			\begin{enumerate}
				\item \(j=i\).
				\item \(B_{i}=SBSA(H)\).
				\item \(\mathcal{F}_{B}(\mathfrak{A})=\mathcal{F}_{B}(\mathfrak{A}) \cup B_{i}\).
				\item \(U=V(B_{i})\).
				\item Calculate the \(B_{i}\)-row of \(MC\):
				\[
				MC_{B_{i},v}=
					\left\lbrace 
					\begin{array}{ll}
					(-1) & \text{ if } v \in \Core{B_{i}}\setminus U.\\
					1 & \text{ if } v \in \Supp{}{B_{i}}\setminus U.\\
					0  & \text{ otherwise.}
					\end{array}
					\right.
				\]
			\end{enumerate}
			\item WHILE exist \(v \in  \Supp{}{H}\setminus U\) such that \(deg_{H}(v)=1\) and exist \(l(v) \in \Supp{}{B_{i}}\) such that \(deg_{H}(l(v))=1\) and \(d(v,l(v))=2\):
			\begin{enumerate}
				\item \(i=i+1\).
				\item \(B_{i}=B_{i-1}-l(v)+v\).
				\item \(\mathcal{F}_{B}(\mathfrak{A})=\mathcal{F}_{B}(\mathfrak{A}) \cup \{B_{i}\}\).
				\item \(U=U \cup \{v\}\).
				\item Calculate the \(B_{i}\) row of \(MC\):
				\[
				MC_{B_{i},u}=
					\left\lbrace 
					\begin{array}{ll}
					1 & \text{ if } u=v. \\
					0  & \text{ otherwise.}
					\end{array}
					\right.
				\]
			\end{enumerate}
			\item WHILE exist \(v \in  \Supp{}{H}\setminus U\) such that \(deg_{H}(v)=1\) and exist \(x \in N_{T}(v) \cap \Core{B_{j}}\):
			\begin{enumerate}
				\item CHOSE \(w \in N_{B_{j}}(x)\).
				\item \(i=i+1\).
				\item \(B_{i}=B_{j}(x \rightarrow w)+v\).
				\item \(\mathcal{F}_{B}(\mathfrak{A})=\mathcal{F}_{B}(\mathfrak{A}) \cup \{B_{i}\}\).
				\item \(U=U \cup \{v\}\).
				\item Calculate the \(B_{i}\) row of \(MC\):
				\[
				MC_{B_{i},u}=
					\left\lbrace 
					\begin{array}{ll}
					1 & \text{ if } u=v.\\
					0  & \text{ otherwise.}
					\end{array}
					\right.
				\]
			\end{enumerate}
			\item Forest 1: \(\mathcal{F}1= H -  U\).
			\item Forest 2: \(\mathcal{F}2= \emptyset\)
			\item FOR EACH \(G \in \mathcal{F}1\):
			\begin{enumerate}
				\item Find \(\entrada{G}{B_{j}}\) and \(\salida{G}{B_{j}}\).
				\item Chose \(x \in N_{B_{j}}(\salida{G}{B_{j}}) \).
				\item \(G=G+B_{j}(x \rightarrow \salida{G}{B_{j}})\).
				\item \(\mathcal{F}2= \mathcal{F}2 \cup \{G\}\).
			\end{enumerate}
			\item \(\mathcal{F}= \mathcal{F}2 \cup \left(\mathcal{F} \setminus {H} \right) \).
		\end{enumerate}
		\item OUTPUT: \(\mathcal{F}_{B}(\mathfrak{A})\).
	\end{enumerate}
\end{algorithm}
It is clear that the result of the algorithm is a forest of S-basic subtrees of \(\mathfrak{A}\). We write \(\mathcal{F}_{B}(\mathfrak{A})\) for the forest of S-basic subtrees of \(\mathfrak{A}\) that is obtained when we apply the S-basis forest algorithm to the S-atom \(\mathfrak{A}\).

The matrix \(MC\) was built in order to prove the following result:
\begin{lemma} \label{NullLemma}
	Let \(\mathfrak{A}\) be an S-atom 
	. Then
	\[
	|\mathcal{F}_{B}(\mathfrak{A})|=\supp{}{\mathfrak{A}}-\core{\mathfrak{A}}
	\]
\end{lemma}

\begin{proof}
	Just double counting over the entries of \(MC\).
\end{proof}

\begin{theorem}
	Let \(\mathfrak{A}\) be an S-atom. The set
	\[
	\overrightarrow{\mathcal{F}_{B}(\mathfrak{A})}:=\{\up{\overrightarrow{\mathfrak{B}}}{$\mathfrak{A}$}{$\mathfrak{B}$} \; : \mathfrak{B} \in \mathcal{F}_{B}(\mathfrak{A}) \}
	\]
	called a \textbf{forest basis} of \(\mathfrak{A}\), is a base of \(\N{\mathfrak{A}}\).
\end{theorem}

\begin{proof}
	Clearly \(\overrightarrow{\mathcal{F}_{B}(\mathfrak{A})}\) is a nonempty set of  linearly independent vectors. Besides, given  \(\overrightarrow{\mathfrak{B}}\) the vector associated to \(\mathfrak{B} \in \mathcal{F}_{B}(\mathfrak{A})\)
\[
A(\mathfrak{A})\up{\overrightarrow{\mathfrak{B}}}{$\mathfrak{A}$}{$\mathfrak{B}$}=\theta
\]
Hence, by Lemma \ref{NullLemma}, \(\overrightarrow{\mathcal{F}_{B}(\mathfrak{A})}\) is a basis of \(\N{\mathfrak{A}}\).
\end{proof}

Each vector \(\overrightarrow{\mathfrak{B}} \in \overrightarrow{\mathcal{F}_{B}(\mathfrak{A})}\) is associated to some S-basic subtree \(\mathfrak{B} \leq \mathfrak{A}\), we will say that \(\overrightarrow{\mathfrak{B}}\) is a \textbf{S-basic vector of} \(\mathfrak{A}\).

\begin{corollary}
	Let \(S\) be an S-tree. Then
	\[
	\up{\bigcup_{\mathfrak{A} \in \mathcal{F}_{A}(S)} \overrightarrow{\mathcal{F}_{B}(\mathfrak{A})}}{S}{$\mathfrak{A}$}
	\]
	is a base of \(\N{S}\).
\end{corollary}

\begin{corollary}
	Let \(T\) be a tree. Then
	\[
	\up{\bigcup_{S \in \mathcal{F}_{S}(T)} \left(
	\up{\bigcup_{\mathfrak{A} \in \mathcal{F}_{A}(S)} \overrightarrow{\mathcal{F}_{B}(\mathfrak{A})}}{S}{$\mathfrak{A}$}
	\right) }{T}{S}
	\]
	is a base of \(\N{T}\).
\end{corollary}

Thus
\[
\{e_{2}-e_{3}, e_{10}-e_{11}, e_{10}-e_{12},e_{6}-e_{7}+e_{8}\}
\]
is a basis of \(\N{T}\) for the tree \(T\) in Figure 2.

\section*{Acknowledgement}
The authors are grateful indebted to Professor Juan Manuel Alonso (UNSL) for their active interest in the publication of this paper. Even though the text does not reflect it, we carry on many numerical experiments on \cite{sage}. They gives us the insight for this paper. Aus dem Paradies, das SageMathcloud uns geschaffen, soll uns niemand vertreiben k\"{o}nnen.
 
\section*{}
Funding: This work was partially supported by the Universidad Nacional de San Luis [grant PROIPRO 03-2216] and Secretaria de Pol\'{\i}ticas Universitarias, Ministerio de Educaci\'{o}n, Rep\'{u}blica Argentina, Programa Redes Interuniversitarias IX: "Red Argentino-Brasile\~{n}a de Teor\'{\i}a Algebraica y Algor\'{\i}tmica de Grafos, Etapa 2016". RESOL-2016-1968-E-APN-SECPU-ME.

\section*{References}

\bibliographystyle{apalike}

\bibliography{TAGcitas}

\end{document}